\documentclass[12pt]{amsart}
\usepackage{amsmath,amssymb,amsfonts,amsthm,amsopn}
\usepackage{graphicx,tikz}
\usepackage[all]{xy}
\usepackage{amsmath}
\usepackage{multirow, longtable, makecell, caption, array,enumitem}
\usepackage{amssymb}
\usepackage{mathtools}
\mathtoolsset{showonlyrefs}

\usepackage{bookmark}
\usepackage{hyperref,color}
\calclayout

\usepackage{geometry}
\newcommand{\adm}[1]{{\left\vert\kern-0.25ex\left\vert\kern-0.25ex\left\vert #1 
		\right\vert\kern-0.25ex\right\vert\kern-0.25ex\right\vert}}

\newtheorem{theorem}{Theorem}[section]
\newtheorem{lemma}[theorem]{Lemma}
\newtheorem{corollary}[theorem]{Corollary}
\newtheorem{proposition}[theorem]{Proposition}
\newtheorem{definition}[theorem]{Definition}

\theoremstyle{definition}

\newtheorem{example}[theorem]{Example}
\newtheorem{remark}[theorem]{Remark}

\theoremstyle{definition}

\usepackage{cellspace} %
\def\bR{\mathbb{R}}
\def\bC{\mathbb{C}}
\def\bN{\mathbb{N}}

\def\smo{\setminus\{0\}}

\def\cF{\mathcal{F}}
\def\cS{\mathcal{S}}

\def\cG{\mathcal{G}}

\def\Fm{\mathcal{M}}

\def\cH{\mathcal{H}}

\def\Ep{\mathcal{E}}
\def\cC{\mathcal{C}}

\def\rd{\bR^d}
\def\rn{\bR^n}

\def\rdd{\bR^{2d}}

\def\lan{\langle}
\def\ran{\rangle}

\def\S0{S^0_{0,0}}

\def\Bd'{B_{\delta'}}

\def\cBd'{\bar{B}_{\delta'}}

\def\vp{\varphi}
\def\veps{\varepsilon}
\def\wh{\widehat}
\def\ird{\int_{\rd}}
\def\irn{\int_{\rn}}
\def\vird{\widetilde{\int_{\rd}}}
\def\irdd{\int_{\rdd}}
\def\smo{\setminus\{0\}}
\def\hbi{\frac{i}{\hbar}}
\def\hbid{\frac{i}{2\hbar}}
\def\twph{(2\pi\hbar)}
\def\twpih{(2\pi i \hbar)}

\def\Frd{\mathrm{Fr}(\rd)}
\def\Sjo{M^{\infty,1}}
\def\Sjog{M^{\infty,1}_{(g)}}

\def\Sjord{M^{\infty,1}(\rd)}
\def\Sjorn{\Sjo(\rn)}
\def\vpe{\vp_{\veps}}
\def\Fp{F_+}
\def\Fm{F_-}
\def\cM{\mathcal{M}}
\def\cV{\mathcal{V}}
\def\gau{\tilde{g}}
\def\g{\gamma}
\def\fSjo{W^{\infty,1}}
\def\fSjord{W^{\infty,1}(\rd)}
\def\Miun{{M_{(g_n)}^{\infty,1}(\bR^n)}}
\def\Mium{{M_{(g_m)}^{\infty,1}(\bR^m)}}
\def\Miuinf{{M_{(g)}^{\infty,1}(\bR^\infty)}}
\newcommand\numberthis{\addtocounter{equation}{1}\tag{\theequation}}
\def\Lmin{L_{\mathrm{min}}}
\def\bk{{\bf k}}
\def\Cinftyb{C^\infty_{\mathrm{b}}}
\def\Ckb{C^k_{\mathrm{b}}}

\def\Xint#1{\mathchoice
{\XXint\displaystyle\textstyle{#1}}%
{\XXint\textstyle\scriptstyle{#1}}%
{\XXint\scriptstyle\scriptscriptstyle{#1}}%
{\XXint\scriptscriptstyle\scriptscriptstyle{#1}}%
\!\int}
\def\XXint#1#2#3{{\setbox0=\hbox{$#1{#2#3}{\int}$ }
\vcenter{\hbox{$#2#3$ }}\kern-.6\wd0}}

\def\dashint{\Xint-}

\begin{document}
	
	\title[Phase space analysis of Fresnel integrals]{Phase space analysis of finite and infinite dimensional Fresnel integrals}

	\author[S. Mazzucchi]{Sonia Mazzucchi}
	\address{Dipartimento di Matematica, Università di Trento, Via Sommarive 14, 38123 Povo, Italy}
	\email{sonia.mazzucchi@unitn.it}

	\author[F. Nicola]{Fabio Nicola}
	\address{Dipartimento di Scienze Matematiche ``G. L. Lagrange'', Politecnico di Torino, Corso Duca degli Abruzzi 24, 10129 Torino, Italy}
	\email{fabio.nicola@polito.it}

	\author[S.I. Trapasso]{S. Ivan Trapasso}
	\address{Dipartimento di Scienze Matematiche ``G. L. Lagrange'', Politecnico di Torino, Corso Duca degli Abruzzi 24, 10129 Torino, Italy}
	\email{salvatore.trapasso@polito.it}
	
	
	\begin{abstract}
		The full characterization of the class of Fresnel integrable functions is an open problem in functional analysis, with significant applications to mathematical physics (Feynman path integrals) and the analysis of the Schr\"odinger equation. In finite dimension, we prove the Fresnel integrability of functions in the Sj\"ostrand class $M^{\infty,1}$ --- a family of continuous and bounded functions, locally enjoying the mild regularity of the Fourier transform of an integrable function. This result broadly extends the current knowledge on the Fresnel integrability of Fourier transforms of finite complex measures, and relies upon ideas and techniques of Gabor wave packet analysis. We also discuss infinite-dimensional extensions of this result. In this connection, we extend and make more concrete the general framework of projective functional extensions introduced by Albeverio and Mazzucchi. In particular, we obtain a concrete example of a continuous linear functional on an infinite-dimensional space beyond the class of Fresnel integrable functions. As an interesting byproduct, we obtain a sharp $M^{\infty,1} \to L^\infty$ operator norm bound for the free Schr\"odinger evolution operator.
  \end{abstract}
	\maketitle
	
\section{Introduction}
\subsection{The problem of Fresnel integrability} \label{sec fresnel intro}
The theory of oscillatory integrals is a classical research topic in harmonic analysis \cite{stein}, with countless connections to different branches of pure and applied mathematics. Originally introduced by Stokes and Kelvin in the context of wave diffraction, oscillatory integrals were extensively studied by H\"ormander in relation to the microlocal analysis of Fourier integral operators \cite{Hor71,HorBook}, resulting in a rich and elegant framework with significant applications to regularity theory for partial differential equations and the propagation of singularities \cite{Dui}.

\textit{Fresnel integrals} are oscillatory integrals (of the first kind) with quadratic phase, namely of the form 
\begin{equation}\label{FId-int}
     \int_{\rd} e^{\hbid |x|^2}f(x) dx,
\end{equation}
where $f \colon \bR\to \bC$ is a Borel measurable function and $\hbar$ a positive parameter, playing here the role of inverse frequency. According to H\"ormander's approach, such an integral must be computed by means of a suitable regularization procedure. To be precise, $f$ is said to be a \textit{Fresnel integrable function} if, for all $\vp \in \cS(\rd)$ with $\vp(0)=1$, the limit 
\begin{equation}  \vird e^{\hbid |x|^2}f(x) dx \coloneqq \lim_{\varepsilon \downarrow 0} \, \twpih^{-d/2} \ird e^{\hbid |x|^2}f(x)\vp(\varepsilon x) dx
\end{equation} exists and does not depend on $\vp$. We shall denote by $\Frd$ the space of Fresnel integrable functions on $\rd$.

As a consequence, while $\Frd$ can be proved to be much larger than $L^1(\bR^d)$ and to include several elements that are important in applications, its full characterization is still an open problem. Partial results in this direction include the Fresnel integrability of the so-called \textit{H\"ormander's symbols}, namely $C^\infty$ functions whose derivatives have a suitably controlled growth at infinity \cite{BBR, HorBook}. More recent results \cite{AlBr,ELT,mazz_book} rely on techniques of Fourier analysis and provide the Fresnel integrability of all the functions belonging to the Banach algebra $\cF\cM(\rd)$ of Fourier transforms of finite (i.e., with bounded total variation) complex Borel measures on $\rd$. A Parseval-type representation formula for the Fresnel integral plays a central role in this context: if $f(x) = \int_{\rd} e^{i x \cdot y} d\mu(y)\in \cF\cM(\rd)$ is the Fourier transform of the measure $\mu$, then one has 
\begin{equation}\label{eq-parsev-intro}
        \vird e^{\hbid |x|^2}f(x) dx=\int_{\rd }e^{-\frac{i\hbar}{2}|x|^2}d\mu (x).
\end{equation} More details can be found in Section \ref{sez-fresnel-findim} below. 

This result is particularly significant since it paves the way to the extension of the theory of oscillatory integrals to the case where the underlying integration domain $\rd$ is replaced by an infinite-dimensional space. Besides the relevance of this problem in the framework of abstract integration theory, infinite-dimensional Fresnel integrals are a key ingredient of the mathematical theory of \textit{Feynman path integrals}. Introduced in the 1940s \cite{feyn1,feyn2}, Feynman path integrals are heuristic representations of the solutions of the Schr\"odinger equation in terms of oscillatory integrals over an infinite-dimensional space of paths. The rigorous analysis of path integrals is a longstanding challenge for the mathematical community, due to several possible approaches and the diverse toolset of techniques needed to encompass even the most basic physical models \cite{AlHKMa,fuji_book,mazz_book}. 

In fact, from a mathematical point of view, it would be interesting to give a meaning to expressions like \eqref{FId-int} where $\bR^d$ is replaced by an infinite dimensional Hilbert space. This is not an easy task, as witnessed by the failure \cite{Cam,Tho} of traditional techniques of infinite dimensional integration based on the Kolmogorov existence theorem --- which allows for the construction of probability measures on infinite dimensional spaces as \textit{projective limits of probability measures} starting from their values on finite dimensional subspaces \cite{Boc,Bau}. Indeed, as detailed in \cite{AlMa2016,Tho}, these techniques cannot be readily extended to the case where probability measures are replaced by complex measures in order to obtain one of the latter whose finite dimensional approximations are Gaussian measures with complex covariance of the form $(2\pi z)^{d/2}e^{-\frac{|x|^2}{2z}}dx$, with $z\in \bC$, $\mathrm{Re}(z)\geq 0$ and $\mathrm{Im}(z)\neq 0$, since the resulting complex measure would have infinite total variation. 

The no-go results just discussed imply that the construction of infinite dimensional oscillatory integrals calls for alternatives to the standard Lebesgue theory, in order to  circumvent the lack of an underlying ``flat'' measure similar to the Lebesgue measure in finite dimensional spaces. A more flexible approach, whose roots are to be found in Daniell's integration and the Riesz-Markov-Kakutani representation theorem, provides that integration should be developed in terms of linear continuous functionals defined on a suitable class of functions. The theory of infinite dimensional Fresnel integrals by Albeverio and H{\o}egh-Krohn precisely embraces this point of view \cite{AlHKMa}. In their work, Fresnel integrals on a (real, separable) infinite dimensional Hilbert space $(\cH, (\cdot,\cdot) )$ are designed using a linear continuous functional $\ell$ on the Banach algebra $\cF\cM(\cH)$ of Fourier transforms of complex Borel measures on $\cH$. More precisely, the definition is given again in terms of a generalized Parseval-type equality: if $f(x)=\int_{\cH} e^{i (x,y)} d\mu(y) \in \cF\cM(\cH)$, then we set (following \cite{AlHKMa})
\begin{equation}
    \ell(f) = \widetilde{ \int_\cH}e^{\hbid \|x\|^2}f(x) dx \coloneqq \int_\cH e^{-\frac{i\hbar}{2}\|x\|^2}d\mu(x). 
\end{equation}
The integration theory developed in this way (see also \cite{AlBr,ELT} for related extensions) is quite robust and still encompasses most of the properties that hold in the finite dimensional context, yet leaving open the problem of the full characterization of the largest class of Fresnel integrable functions $\mathrm{Fr}(\cH)$. Several attempts have been made to prove Fresnel integrability results beyond the $\cF\cM(\cH)$ scenario, both in finite and infinite dimensions \cite{AlMa2005}, but a general theory is currently unavailable and only few examples have been proposed so far --- which are ultimately related to particular physical models and special representations formulae based on an analytic continuation technique \cite{AlMa2,AlCaMa}.

\subsection{Fresnel integrability for the Sj\"ostrand class}
In this note we address the problem of Fresnel integrability, in both finite and infinite dimensional settings, of a broader class of functions, which in the finite dimensional case is strictly larger than $\cF\cM(\rd)$. More precisely, in $\rd$ we consider the \textit{Sj\"ostrand class} $\Sjord$ of temperate distributions $f \in \cS'(\rd)$ satisfying the following condition: given a Schwartz function $g \in \cS(\rd)\smo$, called \textit{window}, we have that $f \in \Sjord$ if and only if 
\begin{equation}\label{eq-sjo-intro}
    \|f\|_{M^{\infty,1}_{(g)}} \coloneqq \ird \sup_{x\in \rd} |\cF (f \overline{g(\cdot - x)}) (\xi)| < \infty, 
\end{equation}
where $\cF$ stands for the Fourier transform and $\overline{v}$ for the complex conjugate of $v \in \bC$. The expression above defines the $M^{\infty,1}(\bR^d)$-norm of $f$ and turns $\Sjord$ into a Banach space. It is then not difficult to realize (see Section \ref{sec-sjo} for details) that any $f\in M^{\infty,1}(\bR^d)$ is a continuous, bounded function whose Fourier transform roughly behaves like a $L^1$ function --- in particular, $f$ locally enjoys the same weak regularity of the Fourier transform of an integrable function. 

A condition similar to \eqref{eq-sjo-intro} appeared in Sj\"ostrand's paper \cite{sjo} to design an unconventional class of symbols whose quantization is stable to composition and still leads to bounded pseudodifferential operators on $L^2(\rd)$. In fact, conditions in the same spirit already appeared in the pioneering works by Feichtinger on \textit{modulation spaces} \cite{fei_mod83}, since the regularity of $f$ is characterized by means of a mixed $L^\infty-L^1$ constraint on a \textit{phase space representation} of $f$ obtained by the \textit{windowed} Fourier transform $V_gf(x,\xi) = \cF(f \overline{g(\cdot-x)})(\xi)$, relative to $g$, also known as the \textit{Gabor transform} or \textit{short-time Fourier transform} in the time-frequency analysis literature \cite{gro_book}. The choice of this special function space is not fortuitous, since it already played a significant role in some recent advances on mathematical path integrals powered by techniques and ideas of Gabor analysis \cite{FNT, NT_cmp,NT_jmp,T_imrn} --- see also \cite{NT_book} for a broad overview of this approach. 

It was already proved in \cite{NT_cmp} that $\cF \cM(\rd) \subsetneq \Sjord$, and in the same reference it is also recalled that $\Cinftyb(\rd) \subset \Sjord$, where $\Cinftyb(\rd)$ is the space of smooth, bounded functions with bounded derivatives of any order --- see Example \ref{example-embed} below for explicit instances of a function in $\Cinftyb(\rd) \setminus \cF\cM(\rd)$. Our first main result is Theorem \ref{thm:fr-int-sjo}, where we prove that $\Sjord \subset \Frd$ and also provide a phase space generalization of the Parseval-like formula in \eqref{eq-parsev-intro} --- roughly speaking, the Fourier transform is now replaced by the Gabor transform. Such a representation formula is flexible enough to allow for the extension of the notion of Fresnel integrability to \textit{distributions} in $\cF\Sjord$, a topic that is discussed in Section \ref{sec-int-fousjo}.   

These results rely on a careful analysis of the time-frequency features of the Fresnel function $e^{\hbid |x|^2}$ (see Section \ref{sez-fresnel-findim}), which may be of independent interest --- we plan to investigate the analogous properties of higher-order Fresnel functions like $e^{\hbid |x|^k}$ with $k>2$ in a forthcoming paper. 

\subsection{Infinite dimensional Fresnel integrals at the Sj\"ostrand regularity}

In the second part of the paper we discuss some results concerning the Fresnel integrability of functions in the Sj\"ostrand class in the infinite dimensional setting where $\rd$ is replaced by $\bR^\infty = \bR^{\bN}$. To this aim, we resort to the general theory of \textit{projective systems of functionals and their limits} \cite{AlMa2016}, that generalizes the Daniell-Kolmogorov theory of projective limits of probability measures \cite{Bau} and unifies oscillatory and probabilistic integration in infinite dimensional spaces. All the technical preliminaries are reviewed in Section \ref{sez-psf} and the application to Fresnel integrals is discussed at length in Section \ref{sez-inf-dim-ext} for the benefit of the reader. 

With the aid of a suitable sequence of Gaussian window functions $(g_n)_n$, we extend the definition of the $M^{\infty,1}(\bR^d)$-norm in \eqref{eq-sjo-intro} to an infinite dimensional setting by constructing a $M^{\infty,1}(\bR^\infty)$-norm defined on a subset $\cC_0$ of \textit{cylinder functions} $f \colon \bR^\infty\to \bC$, i.e., those maps which depend explicitly only on a finite number of variables. Furthermore, we show that it is possible to unambiguously define a bounded linear continuous functional $L_{\text{min}} \colon \cC_0\to \bC$ in terms of the finite dimensional Fresnel integrals on $M^{\infty,1}(\bR^n)$ studied in Section \ref{sez-fresnel-findim}. The backbone of this result is a detailed analysis of the family of functionals $L_n \colon \Sjo(\rn) \to \bC$ with respect to the dimension parameter $n \in \bN$, where 
\begin{equation}
    L_n(f) = \widetilde{\int_{\bR^n}}e^{\hbid\|x\|^2}f(x)dx. 
\end{equation} 
In Theorem \ref{th:norm:Ln} we were able to compute the exact norm of these operators, hence obtaining sufficient conditions such that they are uniformly bounded with respect to $n$. 

As an interesting byproduct, in  Section \ref{sec-schro} we discuss how this result can be used to compute the {\it exact norm} of the free Schr\"odinger propagator as an operator from $M^{\infty,1}(\mathbb{R}^n)$ (with Gaussian window) to $L^\infty(\mathbb{R}^n)$. It is indeed well known that the free particle propagator is not bounded $L^\infty(\mathbb{R}^n)\to L^\infty(\mathbb{R}^n)$, hence leading one to replace $L^\infty(\mathbb{R}^n)$ with a smaller space ($M^{\infty,1}(\mathbb{R}^n)$, in this case), according to a common practice in harmonic analysis, where several classical $L^p\to L^p$ inequalities fail to hold in the endpoint cases 
 $p=1,\infty$ and one thus replaces $L^1(\mathbb{R}^n)$ with the Hardy space $H^1$ (or weak-$L^1$) and $L^\infty$ with the space $\text{BMO}$ of functions of bounded mean oscillation --- see \cite{stein} for additional details. 

The aforementioned results are further developed in Section \ref{sez-proj-ext-top}, where we introduce an extension of $L_{\text{min}} \colon \cC_0\to \bC$ in terms of a linear continuous functional $L \colon \bar\cC_0\to \bC$ defined on the closure $\bar \cC_0$ of the set $\cC_0$ of cylinder functions in the $M^{\infty,1}(\bR^\infty)$-norm. A detailed analysis of the elements in the domain $\bar\cC_0$ of $L$ is provided, along with some examples.

Finally, in Section \ref{sez-proj-est-seq} we present a sequential approach to design an alternative linear functional $L' \colon D(L')\to \bC$ that eventually results in a larger extension of $L_{\text{min}}$. In particular, we prove that the domain $\bar \cC_0$ of $L$ and the Albeverio--H{\o}egh-Krohn class $\cF\cM(\ell^2) $ of Fourier transforms of complex Borel measures on $ \ell^2$ are both strictly included in the domain $D(L')$ of $L'$,  the latter thus providing a non-trivial extension of the class of Fresnel integrable function in infinite dimensions. A detailed description of some examples is given in order to concretely substantiate this result.

\subsection*{Summary of the main results}
We provide here, for the convenience of the reader, a short list of the most relevant results contained in this note.  

\begin{enumerate}
    \item Theorem \ref{thm:fr-int-sjo}, concerning the Fresnel integrability of functions in the finite-dimen\-sional Sj\"ostrand class $\Sjord$, along with its applications (Example \ref{example-embed}) and generalizations (Theorem \ref{thm:fr-int-fsjo}). 
    \item Theorem \ref{th:norm:Ln}, on the exact norm of linear continuous Fresnel-type functionals on the Sj\"ostrand space $\Sjo(\rn)$.
    \item Theorem \ref{thm-schro}, where we obtain sharp $L^\infty$ bounds for the solution of the free Schr\"odinger equation with initial datum in $\Sjorn$.  
    \item Example \ref{examplef-nCauchy} focuses on the characterization of the domain of a non-trivial extension $(L,D(L))$ of $(L_{\min}, D(L_{\min}))$. To this aim, the trace-type result in Proposition \ref{Prop-comparison-norm} (with sharp bound \eqref{eq formula uno}) and Theorem \ref{thm-indep-rep} play a key role. 
    \item In Examples \ref{example-1} and \ref{example-2} we discuss the difficulties related to integration of non-cylinder functions in the Fresnel algebra $\cF(\ell^2(\bN))$.
    \item The rest of the paper is devoted to extending the integration domains beyond $\cF(\ell^2(\bN))$ and $\cF \cM(\ell^2(\bN))$, which is obtained by constructing a different linear continuous functional $(L',D(L'))$ that extends $(L,D(L))$ --- cf.\ Theorem \ref{thm-L'-ext}. A concrete function $f \colon \bR^\infty \to \bC$ in $D(L')$ is exhibited, along with useful representation formulae (Lemma \ref{lem:rep_infou}).     
\end{enumerate}

To the best of our knowledge, the results discussed so far substantiate the general abstract framework developed in \cite{AlMa2016} with novel concrete examples, unraveling new classes of Fresnel integrable functions and allowing for the first steps of Gabor analysis beyond the finite dimensional Euclidean setting.

\section{Preliminary materials}

\subsection{Notation}\label{sec notation}
In this note we choose the following convention for the inner product on $L^2(\rd)$:
\begin{equation}    \lan f,g \ran = \ird f(y) \overline{g(y)}dy, \qquad f,g \in L^2(\rd).
\end{equation}
The duality pairing between a temperate distribution $f \in \cS'(\rd)$ and a function $g \in \cS(\rd)$ in the Schwartz class is still denoted by $\lan f,g \ran$, upon agreeing that $\lan f,g \ran = f(\overline{g})$ for the sake of consistency. 

Given the parameter $\hbar \in (0,1]$, the Fourier transform of $f \in L^2(\rd)$ is defined here by setting
\[  \hat{f}(\xi)=\cF(f)(\xi) = \twph^{-d/2} \ird e^{-\hbi \xi\cdot x}f(x)dx, \quad \cF^{-1}(f)(x) = \twph^{-d/2} \ird e^{\hbi x\cdot \xi}f(\xi)d\xi, \]
with obvious extensions to temperate distributions. The Parseval identity thus reads $\lan f,g \ran = \lan \hat{f},\hat{g} \ran$.  

Given $\lambda \in \bR\smo$ and $f \colon \rd \to \bC$, we set $f\circ \lambda (x) \coloneqq f(\lambda x)$. 

Given $A,B \in \bR$, we occasionally write $A \lesssim B$ to mean that the inequality $A \le C B$ holds, where $C>0$ is a constant that does not depend on $A,B$ --- but may depend on the parameter $\lambda$, in which case we write $A \lesssim_\lambda B$. 

The inhomogeneous magnitude of $y \in \rd$ is defined as follows: $\lan y \ran \coloneqq (1+|y|^2)^{1/2}$.

The symbol $\ast$ denotes convolution, as usual.

\subsection{Basic toolkit of Gabor analysis} \label{sec gabor} Let us briefly review some results of time-frequency analysis that will be used below. The reader is addressed to \cite{gro_book,NT_book} for proofs and further details. 

Given $x \in \rd$ and $\xi \in \wh\rd \simeq \rd$, the translation and modulation operators $T_x$ and $M_x$, acting on $f \colon \rd \to \bC$, are defined by
\begin{equation}
    T_x f(y) \coloneqq f(y-x), \qquad M_\xi f(y) \coloneqq \twph^{-d/2} e^{\hbi \xi \cdot y}f(y), \qquad y \in \rd.
\end{equation}
The phase space shift of $f$ along $(x,\xi) \in \rd \times \wh{\rd} \simeq \rdd$ is thus defined by $\pi(x,\xi) \coloneqq M_\xi T_x$. 

The Gabor/short-time Fourier transform (STFT) of $f \in \cS'(\rd)$ with respect to the window $g \in \cS(\rd)\smo$ is defined by
\begin{equation}
    V_g f(x,\xi) \coloneqq \lan f,\pi(x,\xi) g \ran = \cF (f \cdot \overline{T_x g})(\xi), \qquad (x,\xi) \in \rdd.
\end{equation}
We also set, for later convenience, 
\begin{equation} 
    \cV_g f(x,\xi) \coloneqq \overline{V_g \overline{f}}(x,\xi) = \twph^{-d/2} \ird e^{\hbi \xi \cdot y} f(y) g(y-x) dy = V_{\bar{g}}f(x,-\xi). 
\end{equation}
It is not difficult to show that $V_g f$ is a continuous function in $\rdd$ with at most polynomial growth (that is $|V_g(z)| \le C (1+|z|)^N$ for some $C>0$ and $N \in \bN$, with $z \in \rdd$). We recall for later use a couple of basic properties of the STFT:
\begin{itemize}
    \item The Fourier transform relates with a right-angle rotation in phase space, according to the pointwise identity
    \begin{equation}\label{eq:fund-id-stft}
        V_g f(x,\xi) = e^{-\hbi x\cdot \xi} V_{\hat g} \hat f (\xi,-x).
    \end{equation}
    \item The following inversion formula for the Gabor transform: if $f \in \cS'(\rd)$ and $g,\gamma \in \cS(\rd)$ are such that $\lan \gamma,g \ran \ne 0$, we have
\begin{equation}\label{eq:invers-stft}
    f = \frac{1}{\lan \gamma,g \ran} \irdd V_g f (z)\pi(z)\gamma dz, 
\end{equation} where the identity holds in the sense of distributions. 
\end{itemize}

It is intuitively clear that the regularity of a function/distribution can be measured in terms of the decay/summability of its phase space representations. This principle motivates the introduction of the so-called \textit{modulation spaces}. To be precise, the modulation space $M^{p,q}(\rd)$ with $1 \le p,q \le \infty$ contains all the distributions $f \in \cS'(\rd)$ such that, for a given $g \in \cS(\rd)\smo$, the mixed Lebesgue norm
\begin{equation}\label{eq:modsp-norm}
    \| f \|_{M^{p,q}_{(g)}} \coloneqq \| V_g f (x,\xi) \|_{L^q_\xi (L^p_x)} = \Big( \ird \Big( \ird |V_g f(x,\xi)|^p dx \Big)^{q/p} d\xi \Big)^{1/q}
\end{equation} is finite --- with obvious modifications if $p,q=\infty$. As a rule of thumb, $f$ belongs to $M^{p,q}$ if it roughly behaves as an $L^p$ function, while its Fourier transform $\hat{f}$ loosely enjoys the regularity of a $L^q$ function. 

It turns out that $(M^{p,q}(\rd), \|\cdot\|_{M^{p,q}_{(g)}})$ is a Banach space for every $p,q$, and the norms $\|\cdot \|_{M^{p,q}_{(g)}}$ and $\|\cdot \|_{M^{p,q}_{(\gamma)}}$ are equivalent for every choice of a different window $\gamma \in \cS(\rd)\smo$. We also have $\cS(\rd)\subset M^{p,q}(\rd)$, with dense embedding if $p,q < \infty$. Remarkably, $M^{2,2}(\rd)$ coincides with $L^2(\rd)$. Moreover, modulation spaces are increasing with the indices, meaning that $M^{p_1,q_1}(\rd) \subseteq M^{p_2,q_2}(\rd)$ as long as $p_1\le p_2$ and $q_1\le q_2$. 

As far as duality is concerned, the form $\lan f,g \ran$ in $\cS \times \cS$ extends to modulation spaces as follows: if $f \in M^{p,q}(\rd)$ and $h \in M^{p',q'}(\rd)$, where $p',q'$ are the H\"older conjugate indices of $p,q$ respectively (hence satisfy $1/p+1/p'=1/q+1/q'=1$), then for all $g,\gamma \in \cS(\rd)$ such that $\lan g,\g \ran \ne 0$ we have
\begin{equation} \label{eq dual star}
    \lan h,f \ran_* \coloneqq \frac{1}{\lan g,\g \ran} \irdd V_\g h(x,\xi) \overline{V_g f(x,\xi)} dxd\xi.
\end{equation}
Note in particular that $\lan g,\g \ran \lan h, \overline{f} \ran_* = \irdd V_\g h(x,\xi) \cV_g f(x,\xi) dx d\xi.$

Swapping the integration order in \eqref{eq:modsp-norm} is linked to the Fourier transform by means of the identity \eqref{eq:fund-id-stft}. To be precise, setting
\begin{equation}\label{eq:wasp-norm}
    \| f \|_{W^{p,q}_{(g)}} \coloneqq \| V_g f (x,\xi) \|_{L^q_x (L^q_\xi)} = \Big( \ird \Big( \ird |V_g f(x,\xi)|^p d\xi \Big)^{q/p} dx \Big)^{1/q}, 
\end{equation}
it is easy to see that $\|f\|_{M^{p,q}_{(g)}} = \| \hat{f}\|_{W^{p,q}_{(\hat g)}}$. In general, the Wiener amalgam space $W^{p,q}(\rd) = \cF M^{p,q}(\rd)$ contains all the distributions $f \in \cS'(\rd)$ such that $\|f\|_{W^{p,q}_{(g)}}<\infty$ for some (hence any) non-trivial $g \in \cS(\rd)$. 

\subsection{The Sj\"ostrand class}\label{sec-sjo} A distinguished role among modulation spaces is played by $\Sjord$, also known as the Sj\"ostrand class --- after \cite{sjo}, where this space appeared in disguise as a non-standard symbol class for pseudodifferential operators, which actually enjoy a number of nice properties (such as boundedness on $L^2$ and stability under composition and inversion, later revisited from a time-frequency angle in \cite{gro_sjo}). 

Recall that $f \in \Sjord$ if, for some $g \in \cS(\rd)\smo$,
\begin{equation}
    \|f\|_{M^{\infty,1}_{(g)}} = \ird \sup_{x \in \rd} |V_g f(x,\xi)| d\xi < \infty. 
\end{equation}
It is a straightforward consequence of the definition that $f$ is a bounded and continuous function on $\rd$, locally enjoying the regularity of a $\cF L^1$ function. More precisely, since the right-hand side of the inversion formula \eqref{eq:invers-stft} is a well-defined Lebesgue integral, we have a pointwise identity implying boundedness:
\begin{equation}
    \| f \|_{L^\infty} \le \frac{\twph^{-d/2}}{|\lan \gamma,g\ran|} \|\gamma\|_{L^1} \|f\|_{\Sjog}.
\end{equation}
Resorting again to \eqref{eq:invers-stft} shows that $f$ is uniformly continuous on $\rd$, in view of the bound
\begin{equation}
    |f(y+u)-f(y)| \le \frac{\twph^{-d/2}}{|\lan \gamma,g\ran |} \irdd |V_g f(x,\xi)| |e^{-\hbi \xi \cdot u} \overline{\gamma(y+u-x)}-\overline{\gamma(y-x)}| dxd\xi, 
\end{equation} hence vanishing as $u \to 0$ by dominated convergence. 

A well known feature of $\Sjord$ is that it is a Banach algebra for pointwise multiplication \cite{reich-s}. We also recall that two relevant subsets of the Sj\"ostrand class are given by the space $\Cinftyb(\rd)$ consisting of infinitely differentiable bounded functions with bounded derivatives of any order (in fact, $\Ckb(\rd)\subset \Sjord$ holds for $k>d$, see \cite[Theorem 14.5.4]{gro_book}), and the space $\cF \cM (\rd)$ of Fourier transforms of finite complex Borel measures on $\bR^d$ (see \cite[Proposition 3.4]{NT_cmp}). 

We remark for later use that if $(f_n)_n$ is a sequence in $\Sjo$ such that $\|f_n -f\|_{\Sjo} \to 0$ as $n \to \infty$, then $f_n \to f$ uniformly on compact subsets of $\rd$. To prove this fact, note that if $K \subset \rd$ is an arbitrary compact subset and $g_K\in \cS(\rd)$ is a bump function such that $g \equiv 1$ on $K$, then
\begin{equation}
\|(f_n-f)g_K\|_{L^\infty} \le \|\cF ((f_n-f)g_K))(\cdot) \|_{L^1} = \|V_{g_K}(f_n-f)(0,\cdot)\|_{L^1}.
\end{equation}
Nevertheless, the norm topology associated with $\|\cdot \|_{\Sjog}$ can be quite restrictive as far as limiting arguments are concerned. For this reason, we introduce the convenient notion of \textit{narrow convergence}: we say that $f_n \to f$ narrowly if $f_n \to f$ weakly in $\cS'(\rd)$ (hence $V_g f_n \to V_gf$ pointwise on $\rdd$), and there exists a controlling function $h \in L^1(\rd)$ such that, for some (hence any) window $g$,
\begin{equation}
    \sup_{x \in \rd} |V_g f_n (x,\xi)| \le h(\xi), \qquad \text{(Lebesgue) a.e. } \xi \in \rd, 
\end{equation} uniformly with respect to $n$. 

It is proved in \cite{sjo} that $\cS(\rd)$ is densely embedded in $\Sjord$ with respect to the narrow convergence, and also that if $f_n \to f$ narrowly then $f_n(y) \to f(y)$ for every $y \in \rd$ --- this follows easily by combining the inversion formula \eqref{eq:invers-stft} with dominated convergence, since
\begin{equation}
    |f_n(y)-f(y)| \le \frac{\twph^{-d/2}}{|\lan \gamma,g\ran |} \|\gamma\|_{L^1} \ird \sup_{x \in \rd}|V_g (f_n-f)(x,\xi)| d\xi, \qquad y \in \rd.
\end{equation}

\subsection{Infinite dimensional integration and projective systems of functionals}\label{sez-psf}
This section is devoted to a systematic description of a general integration theory on infinite dimensional spaces based on the notion of linear functional. 

Recall that the Riesz-Markov theorem states a one-to-one correspondence between complex Borel measures and linear continuous functionals on the space of continuous functions vanishing at infinity. However, the latter result relies upon the local compactness of the underlying topological space where the Borel measure is defined, a condition that generally is not fulfilled in an infinite dimensional setting. Moreover, Kolmogorov existence theorem (which is one of the cornerstones of modern probability theory, and provides the construction of probability measures on suitable infinite dimensional spaces \cite{Bau} out of their finite dimensional approximations) does not extend trivially to the case of complex or signed measures \cite{Tho}. Indeed, the necessary requirement of finite total variation of the resulting measure can be hardly fulfilled in the most interesting cases --- see \cite{Cam,Hoc} for some particular examples and \cite{Tho,AlMa2016} for a systematic description of these and related issues.

We review here the recent theory of \textit{projective system of functionals} \cite{AlMa2016}, which generalizes the theory of projective systems of measures and their extensions \cite{Boc,Tho}, hence providing a unified setting for infinite dimensional integration of probabilistic and oscillatory type. The construction of the infinite dimensional space, the companion set of integrable functions and the associated functionals will be presented in the general setting of projective families and their limits, generalizing Bochner's approach to Kolmogorov's construction \cite{Boc}.

\subsubsection{Projective limit spaces}
Let us start by considering a \textit{directed set} $A$, i.e., a set equipped with a partial order $\leq$ such that for any $J,K\in A$, there exists $R\in A$ satisfying $J\leq R$ and $K\leq R$.

A {\em projective (or inverse) family of sets} $\{E_J,\pi^K_J\}_{J,K\in A}$ is a collection  $\{E_J\}_{J\in A}$  of (non-empty)  sets $E_J$  labelled by the elements of the directed set $A$ and endowed with a corresponding collection of surjective map $ \pi^K_J \colon E_K\to E_J$, with $J,K\in A$ and $J\leq K$, satisfying the following conditions:
\begin{enumerate}
    \item $\pi_K^K$ is the identity on $E_K$. 
    \item $\pi^R_J=\pi^K_J\circ \pi^R_K$ for  all $J\leq K\leq R$, $R\in A$.
\end{enumerate}  
The {\em projective (or inverse) limit} $E_A\coloneqq \varprojlim E_J$ of the projective family  $\{E_J, \pi^K_J\}_{J,K\in A}$ is defined as the following subset of the Cartesian product of the sets $\{E_J\}_{J\in A}$:
\begin{equation} \label{eq-def-EA} E_A\coloneqq \Big\{(x_J)\in \bigtimes_{J\in A}E_J : x_J=\pi^K_J(x_K)\; \hbox{\rm  for all }J\leq K,\, J,K\in A \Big\}. \end{equation}

The projective family $\{E_J,\pi^K_J\}_{J,K\in A}$ is called {\em topological} if each  $E_J$, $J\in A$, is a topological space and the maps $ \pi^K_J \colon E_K\to E_J$, $J\leq K$, are continuous. In this case $E_A = \varprojlim E_J$ is equipped with the coarsest topology making all the projection maps $\pi_J \colon E_A\to E_J$ continuous. 

\begin{example} \label{example1} A basic example of the construction above is obtained in the case where $A=\bN$ and $E_J=\bR^J$, with $J \in A$.The corresponding projective limit space $E_A$ is then isomorphic to the space $\bR^\bN$ of real-valued sequences, which will be also denoted by $\bR^\infty$ from now on. 
\end{example}

Given the coordinate projection $\tilde \pi_J \colon \bigtimes_{J\in A}E_J\to E_J$, we shall denote by $\pi_J\coloneqq \tilde \pi_J|_{E_A}$ its restriction to $E_A$. It is easy to verity that the projection maps $\{\pi_J\}_{J\in A}$ satisfy 
 \begin{equation}\label{rel-pi}
     \pi_J=\pi^K_J\circ \pi_K\qquad \forall J,K\in A, J\leq K.
 \end{equation} 
 
 \subsubsection{Cylinder functions}\label{sez-cylinder function}
For any $J\in A $ we denote with the symbol $\hat E_J$ the space of complex valued functions $f_J \colon E_J\to \bC$, and similarly with $\hat E_A$ the family of maps $f \colon E_A\to \bC$. For any $J,K\in A$ with $J\leq K$ we can define the extension $\Ep_J^K \colon \hat E_J\to \hat E_K$ by setting
\[ \Ep_J^K(f_J)(\omega_K)\coloneqq f_J\big(\pi_J^K(\omega_K)\big), \qquad f_J\in \hat E_J, \, \omega_K\in E_K.\]
A function $f_J \in \hat E_J$, $J\in A$, can be extended in a similar way to a function $\Ep_J^A f_J\coloneqq \Ep_J^A( f_J)\in \hat E_A$ on the projective limit space $E_A$:
\[\Ep_J^Af_J(\omega)\coloneqq f_J(\pi_J\omega), \qquad \omega \in E_A.\] 
In view of \eqref{rel-pi}, the extension maps $\Ep_J^A \colon \hat E_J\to \hat E_A$ satisfy the following condition for any $J,K\in A$, with $J\leq K$:
 \begin{equation}\label{rel-Ep}\Ep^A_J=\Ep^A_K\circ \Ep_J^K.\end{equation}

The functions on $E_A$ obtained as the extension of some $f_J\in \hat E_J$ will be called {\em cylinder functions} and denoted by the symbol $\cC$:
\[ \cC\coloneqq \bigcup_{J\in A}\Ep_J^A(\hat E_J). \]
In particular, in the case of Example \ref{example1} the cylinder functions are those depending explicitly only on a finite number of variables.

\subsubsection{Projective systems of functionals}
\begin{definition}\label{def-projectivesystemsfunctionals}
    Consider, for every $J\in A$, a linear complex-valued functional $L_J \colon \hat E_J^0\to\bC$ with domain $\hat E_J^0\subset \hat E_J$. The collection of such functionals $\{L_J,\hat E_J^0\}_{J\in A}$ forms a {\em projective system of functionals} if, for all $J,K\in A$ with $J\leq K$, the following projective (or coherence or compatibility) conditions hold:
\begin{enumerate}
\item $\Ep_J^K(f_J)\in \hat E_K^0$, $\forall f_J\in \hat E_J^0$.
\item $(L_K)(\Ep_J^K(f_J))=L_J(f_J)$, $\forall f_J\in \hat E_J^0$.
\end{enumerate}

\end{definition}

\begin{example}\label{example:measurespaces}
    As a prominent example of the construction above, consider the case where each element $(E_J)_{J\in A}$ of the projective family $\{ E_J,\pi^K_J \}$ is endowed with a $\sigma$-algebra $\Sigma_J$ of subsets of $E_J$ and all the projection maps $\{\pi^K_J\}_{J\leq K}$ are measurable. Moreover, each measurable space $(E_J, \Sigma_J)$ is endowed with a finite measure $\mu_J$, and the collection $\{\mu_J\}_{J\in A}$ satisfies the {\em consistency property}
    \begin{equation}\label{consistency:measures}
        \pi^K_J\mu_K=\mu_J\qquad \forall J\leq K,
    \end{equation}
    where $\pi^K_J\mu_K$ stands for the image measure of $\mu_K$ through the action of the map $\pi_J^K$.
    
    As a consequence of the consistency  property \eqref{consistency:measures}, it is easy to verify that the collection of functionals $\{L_J,\hat E_J^0\}_{J\in A}$ defined by
    \begin{equation}\label{proj-syst-measure}
        E_J^0\coloneqq L^1(E_J,\Sigma_J,\mu_J), \qquad L_J(f)\coloneqq \int_{E_J} fd\mu_J, \, f\in \hat E_J^0,
    \end{equation}    
    is a projective system of functionals.

\end{example}
\subsubsection{Extensions of  projective systems of functionals}
We denote by $\cC_0\subset \cC$ the subfamily of cylindrical functions consisting of those maps $f\in \hat E_A$ that are obtained as  extensions  $\Ep_J^Af_J$, for some $J\in A$ and some $f_J \in \hat E^0_J$:
\begin{equation}
    \cC_0\coloneqq \bigcup_{J\in A}\Ep^A_J(\hat E_J^0).
\end{equation}

\begin{definition}\label{def-proj-syst}
A {\em projective extension} $(L,D(L))$ of a projective system of functionals $\{L_J, \hat E_J^0\}_{J\in A}$  is  a functional $L$ with domain $D(L)\subseteq \hat E_A$ such that:
\begin{enumerate}
\item $\cC_0 \subseteq D(L)$.
\item $L(\Ep_J^A f_J)=L_J(f_J)$, for all $f_J\in \hat E_J^0$.
\end{enumerate} 
\end{definition}

If $\{E_J, \pi^K_J\}_{J,K\in A}$ is  a  perfect inverse system\footnote{Recall that a projective family $\{E_J, \pi^K_J\}_{J,K\in A}$ is said to be a {\em perfect inverse system} if for all $J\in A$, $x_J\in E_J$, there exists $x\in E_A$ (cf.\ \eqref{eq-def-EA}) such that $x_J=\pi_J x$. Every projection is thus surjective in this setting. The projective family described in Example \ref{example1} is a perfect inverse system. } then a projective extension  $(L,D(L))$ exists, given by the functional 
 \begin{equation}
D(L) \coloneqq \cC_0, \qquad L(f) \coloneqq  L_J(f_J), \qquad f=\Ep_J^Af_J, \quad f_J\in \hat E_J^0.
\end{equation}
This functional is ``minimal'', meaning that any other extension $( L',D( L'))$ will satisfy (by definition) the conditions $D(L)\subseteq D(L')$ and  $L'(f)=L(f)$ for all $f\in D(L)$. 

Besides the minimal extension, it is interesting to investigate the existence of larger extensions $(L, D(L))$. Consider for instance the setting of Example \ref{example:measurespaces}; if all the measures of the consistent family $\{\mu_J\}_{J\in A}$ are probability measures and the probability spaces satisfy a suitable set of sufficient conditions (e.g., $(E_J,\Sigma_J)$ are topological spaces endowed with the Borel $\sigma$-algebra and the measures $\mu_J$ are inner regular), then the Kolmogorov existence theorem \cite{Boc,Yam} implies the existence of a unique probability measure $\mu$ on the projective limit space $E_A$, equipped with the smallest $\sigma$-algebra $\Sigma_A$ that makes all the projections $\pi_J \colon E_A\to E_J$ measurable, in such a way that 
\begin{equation}\label{cond:measure:mu}
    \mu_J=\pi_J\mu\qquad \forall J\in A.
\end{equation} 
A natural extension of the projective system of functionals \eqref{proj-syst-measure}
is thus given by \begin{equation}\label{proj-ext-measure}
    D(L)\coloneqq L^1(E_A,\Sigma_A,\mu), \quad L(f)\coloneqq \int_{E_A} fd\mu, \qquad f\in D(L).
\end{equation}
However, in the case where $\{\mu_J\}_{J\in A}$ is a consistent family of complex measures, a necessary condition for the existence of a complex measure $\mu \colon \Sigma_A\to \bC$ satisfying condition \eqref{cond:measure:mu} is the existence of a uniform bound on the total variation of the measures belonging to the consistent family, that is:
\begin{equation}\label{cond:totalvariation}
    \sup_{J\in A}|\mu_J|<\infty. 
\end{equation}
In other words, if condition \eqref{cond:totalvariation} is not fulfilled, then a projective extension of the form \eqref{proj-ext-measure} cannot exists --- see \cite{Tho} for further details.

An alternative construction of a projective extension relies on a topological argument. Assume that the space $\hat E_A$ is endowed with a topology $\tau$ such that $(\hat E_A, \tau)$ becomes a topological vector space. The minimal extension is said to be {\em closable} in $\tau$ if the closure of the graph 
\[ \cG (\Lmin)\coloneqq \{ (f,\Lmin(f))\in \hat E_A\times \bC \colon f\in D(\Lmin) \} \]
in $\hat E_A\times \bC$ with respect to the product topology $\tau\times \tau_\bC$ is the graph of a well defined functional --- here $\tau_\bC$ denotes the standard topology in $\bC$. In this case we define the {\em closure of $\Lmin$} in $\tau$ as the functional $\bar L_\tau$ whose graph satisfies $\cG (\bar L_\tau)=\overline{\cG (\Lmin)}$. In particular, if $\Lmin \colon\cC_0\to \bC$ is continuous in the $\tau$ topology on $\cC_0$, then $\Lmin$ is closable --- see \cite{AlMa2016} for additional details.

\section{Fresnel integrability of functions in $M^{\infty,1}$}\label{sez-fresnel-findim}
Let us start by recalling the definition of Fresnel integrability in $\rd$ from Section \ref{sec fresnel intro}.
\begin{definition}\label{def-fresnel-integral}
A function $f \colon \rd \to \bC$ is said to be Fresnel integrable if, for all $\vp \in \cS(\rd)$ with $\vp(0)=1$, the limit 
\begin{equation}\label{eq:fresnel int} \vird e^{\hbid |x|^2}f(x) dx \coloneqq \lim_{\varepsilon \downarrow 0} \, \twpih^{-d/2} \ird e^{\hbid |x|^2}f(x)\vp(\varepsilon x) dx
\end{equation} exists and is independent $\vp$. The class of Fresnel integrable functions on $\rd$ is denoted by $\Frd$  
\end{definition}
As a consequence of the Parseval identity, the function $f=1$ belongs to $\Frd$ and the corresponding Fresnel integral is normalized:
\[ \vird e^{\hbid |x|^2}dx = \lim_{\varepsilon \downarrow 0} \, \twpih^{-d/2} i^{d/2} \ird e^{-\hbid |\veps \xi|^2}\hat{\vp}(\xi) d\xi = \twph^{-d/2} \twph^{d/2} \vp(0) = 1. \]
It was proved in \cite{AlBr,ELT} (see also \cite[Theorem 5.2]{mazz_book}) that $\Frd$ includes the space $\cF\cM(\rd)$ of (inverse) Fourier transforms of finite complex measures on $\rd$, i.e., the maps of the form 
\begin{equation}
    f(x)=\int_{\bR^d}e^{i\xi\cdot x}d\mu_f(\xi), \qquad x\in \rd\, ,
\end{equation}
for some complex Borel measure $\mu_f$ on $\rd$. The space $\cF\cM(\rd)$ is actually a Banach algebra for pointwise multiplication with norm $\| f \|_{\cF\cM} \coloneqq |\mu_f|$, where $|\mu_f|$ is the total variation of the measure $\mu_f$.

The Fresnel integral of $f \in \cF\cM(\rd)$ can be computed in terms of the following Parseval-type equality (see, e.g.\  \cite{AlHKMa,ELT} and \cite[Theorem 5.2]{mazz_book}):
\begin{equation}\label{Parseval-type-rd}
    \vird e^{\hbid |x|^2}f(x) dx=\int_{\rd }e^{-\frac{i\hbar}{2}|x|^2}d\mu (x)\,,
\end{equation}
which is a crucial tool in the extension of the theory to the case where $\bR^d$ is replaced by an infinite dimensional Hilbert space, with relevant applications to the mathematical theory of Feynman path integrals \cite{AlHKMa}. 

Motivated by the embedding $\cF\cM(\rd) \subset \Sjord$ (see \cite[Proposition 3.4]{NT_cmp}) already recalled in Section \ref{sec-sjo}, we plan to extend formula \eqref{Parseval-type-rd} beyond the realm of Fourier transforms of measures. To this aim, a preliminary time-frequency analysis of the Fresnel functions 
\begin{equation}
    F_+(x)\coloneqq \twpih^{-d/2} e^{\hbid |x|^2}, \qquad F_-(x) \coloneqq \twpih^{-d/2} e^{-\hbid |x|^2}.
\end{equation}
is needed. It is easy to realize that $\wh{F_+} = i^{d/2}F_-$ and, thanks to the particular form of the map $F_+$, in the case of a Gaussian window function $g$ the Gabor transform of $F_+$ can be computed explicitly, as shown in the following lemma.

\begin{lemma}\label{lem:fresnel_function}
    Set $\gau(y) \coloneqq (\pi \hbar)^{-d/4} e^{-\frac{1}{2\hbar}|y|^2}$, $y \in \rd$. Then, for every $x, \xi \in \rd$,
    \begin{equation}\label{eq:V_gFp gaussian}
        V_{\gau} \Fp(x,\xi) = (2\pi \hbar)^{-d/2}(\pi \hbar)^{-d/4}(1+i)^{-d/2}e^{\frac{i}{4\hbar}(|x|^2-2x\cdot \xi -|\xi|^2)} e^{-\frac{1}{4\hbar}|x-\xi|^2}. 
    \end{equation} 
    In particular, $F_+,F_- \in (M^{1,\infty}(\rd)\cap W^{1,\infty}(\rd))$, but $\Fp,\Fm \notin \Sjord$.

    More generally, for all $g \in \cS(\rd)$ and $N \in \bN$ there exists $C_N>0$ such that  
    \begin{equation}\label{eq:Fp schw decay}
        |V_g \Fp(x,\xi) | \le C_N \lan x-\xi \ran^{-2N}. 
    \end{equation}
\end{lemma}

\begin{proof}
    Formula \eqref{eq:V_gFp gaussian} can be proved by straightforward calculations, and the function space belonging follows immediately --- see for instance \cite[Theorem 14]{benyi_unimod}. 
Concerning \eqref{eq:Fp schw decay}, we need to prove the boundedness of
\begin{equation}
    \lan x-\xi \ran^{2N} |V_g \Fp(x,\xi) |  = \twph^{-d} \left| \ird \lan x-\xi \ran^{2N} e^{-\hbi \xi \cdot y} e^{\hbid |y|^2} \overline{g(y-x)} dy  \right|. 
\end{equation}
Note that $\lan x-\xi \ran^{2N} = \lan x-y+y-\xi \ran^{2N} = \sum_{|\alpha+\beta| \le 2N} c_{\alpha,\beta} (x-y)^\alpha (y-\xi)^\beta$, for suitable $c_{\alpha,\beta}\ge 0$ with $\alpha, \beta \in \bN^d$. It is then enough to prove that
\begin{equation}
    \left| \ird  (y-\xi)^\beta \phi(y,\xi) (x-y)^\alpha \overline{g(y-x)} dy  \right| \le C_{\alpha,\beta}, 
\end{equation} for some $C_{\alpha,\beta}>0$, where we set $\phi(y,\xi) \coloneqq e^{- \hbi \xi \cdot y} e^{\hbid |y|^2}$. We highlight that
\begin{equation}
    \partial_y^\beta \phi(y,\xi) = \Big(\hbi\Big)^{|\beta|}(y-\xi)^\beta \phi(y,\xi) + \sum_{\gamma < \beta} k_{\gamma} (y-\xi)^\gamma\phi(y,\xi),
\end{equation} for suitable coefficients $k_\gamma \in \bC$, $\gamma \in \bN^d$. Integration by parts yields
\begin{align*}
    & \phantom{=} \left| \ird  (y-\xi)^\beta \phi(y,\xi) (x-y)^\alpha \overline{g(y-x)} dy  \right| \\ & =  \hbar^{|\beta|}  \left| \ird  \phi(y,\xi) \partial^\beta_y \big[(x-y)^\alpha \overline{g(y-x)}\big] dy - \sum_{\gamma < \beta} k_\gamma \ird (y-\xi)^\gamma \phi(y,\xi) (x-y)^\alpha \overline{g(y-x)} dy \right| \\ 
    & \le \ird  \left|  \partial^\beta_y \big[y^\alpha \overline{g(y)}\big] \right| dy  + \left| \sum_{\gamma < \beta} k_\gamma \ird (y-\xi)^\gamma \phi(y,\xi) (x-y)^\alpha \overline{g(y-x)} dy \right| \\ 
    & \le C_{\alpha,\beta}' + \sum_{\gamma < \beta} |k_\gamma| \left| \ird (y-\xi)^\gamma \phi(y,\xi) (x-y)^\alpha \overline{g(y-x)} dy \right|,
\end{align*} for some constant $C_{\alpha,\beta}'>0$, hence a recursive application of the same argument proves the claim. 
\end{proof}
The following lemma provides some technical results to be used later on.
\begin{lemma}\label{lem:boundH}
    Let $f \in \Sjord$ and $g,\varphi \in \cS(\rd)$. For $0 < \veps < 1$, set $\vpe(x) \coloneqq \varphi(\veps y)$ for all $y \in \rd$. There exists $H \in L^1(\rd)$, independent of $\veps$, such that
    \begin{equation}\label{eq:vprod-bound}
    \sup_{x\in \rd}|V_g(f\vpe)(x,\xi)| \le H(\xi), \qquad \xi \in \rdd,
\end{equation}
    \begin{equation}\label{eq:vprod-bound-fsjo}
    \sup_{\xi\in \rd}|V_g(\widehat{f}\vpe)(x,\xi)| \le H(x), \qquad x \in \rdd.
\end{equation}
\end{lemma}

\begin{proof} Let us prove \eqref{eq:vprod-bound} first. Resorting to the Parseval formula it is easy to prove the following identity for the Gabor transform of the product of $h_1,h_2 \in L^\infty(\rd)$, or $h_1 \in \cS'(\rd)$ and $h_2 \in \cS(\rd)$: 
\begin{equation}\label{eq:stft-prod}
V_g(h_1\cdot h_2)(x,\xi)= (V_{g_1} h_1(x,\cdot)*V_{g_2} h_2(x,\cdot))(\xi), \qquad (x,\xi) \in \rdd, \end{equation} where $g_1,g_2 \in \cS(\rd)$ are such that $g=g_1g_2$ --- the existence of two Schwartz functions whose product coincides with a given one of the same type is guaranteed by the factorization properties of such class \cite{voigt}. As a consequence, we infer
\[ \sup_{x\in \rd} |V_g(f\vpe)(x,\xi)| \le \Big( \Big( \sup_{x \in \rd} |V_{g_1} f(x,\cdot)|\Big) * \Big( \sup_{x \in \rd} |V_{g_2} \vpe(x,\cdot)| \Big) \Big) (\xi), \] so that it is enough to prove the existence of $\tilde{H} \in L^1(\rd)$, independent of $\veps$, such that \[ \sup_{x \in \rd} |V_{g_2} \vpe(x,\xi)| \le \tilde{H}(\xi),\] hence \eqref{eq:vprod-bound} follows with $H = H_1 = (\sup_{x\in \rd} |V_{g_1} f(x,\cdot)|)*\tilde{H}$. 

Recall that for any $N \in \bN$ we have the identity 
\[ (1-\Delta_y)^N e^{-\hbi \xi \cdot y} = (1+|\xi/\hbar|^2)^N e^{-\hbi \xi \cdot y}.  \] Therefore, integration by parts yields
\begin{align*} \sup_{x\in \rd} |V_{g_2}(\vpe)(x,\xi)| & = \twph^{-d/2} \sup_{x\in \rd} \Big| \ird e^{-\hbi \xi \cdot y} \vp(\veps y) \overline{g(y-x)}dy \Big| \\
& \le \twph^{-d/2} (1+|\xi/\hbar|^2)^{-N} \sup_{x\in \rd} \ird |(1-\Delta_y)^N \big(\vp(\veps y) \overline{g(y-x)}\big)| dy,
\end{align*} and the latter integral is easily seen to be bounded, uniformly with respect to $\veps \in (0,1)$. The desired result follows after setting $\tilde{H}(\xi) = \twph^{-d/2} C_N (1+|\xi|^2)^{-N}$, which is an integrable function for $N$ sufficiently large. 

The proof of \eqref{eq:vprod-bound-fsjo} is similar. Starting from \eqref{eq:stft-prod}, we have 
\begin{align*}
    \sup_{\xi \in \rd} |V_g(\wh{f} \vpe)(x,\xi)| & = \| V_{g_1} \widehat{f} (x,\cdot)*V_{g_2} \vpe (x,\cdot) \|_{L^\infty} \\
    & \le \| V_{g_1} \widehat{f} (x,\cdot)\|_{L^\infty} \| V_{g_2} \vpe(x,\cdot) \|_{L^1}. 
\end{align*}
If $f \in \Sjord$ then $\wh{f} \in \fSjord$, hence $\| V_{g_1} \widehat{f} (x,\cdot)\|_{L^\infty} \in L^1_x(\rd)$. It is enough to prove that $\sup_{x \in \rd} \| V_{g_2} \vpe(x,\cdot) \|_{L^1}$ is bounded, uniformly with respect to $\veps$, to obtain the claim with $H(x)=H_2(x)= \big(\sup_{x \in \rd} \| V_{g_2} \vpe(x,\cdot) \|_{L^1}\big) \| V_{g_1} \widehat{f} (x,\cdot)\|_{L^\infty}$. In fact, we already proved above that $|V_{g_2}(\vpe)(x,\xi)| \le \twph^{-d/2} C_N (1+|\xi|^2)^{-N}$, hence the claim immediately follows. 
\end{proof}

\subsection{Main result} We are now ready to show the Fresnel-integrability of all functions $f$ belonging to the  Sj\"ostrand class $\Sjord$, along with a novel representation formula of Parseval type in phase space.
\begin{theorem}\label{thm:fr-int-sjo}
Every $f \in \Sjord$ is Fresnel integrable, and for every $g,\g \in \cS(\rd)$ such that $\lan \g, g \ran \ne 0$ we have
\begin{equation}\label{Rep-formula-rd}
    \vird e^{\hbid |x|^2}f(x) dx = \frac{1}{\lan \g,g \ran} \irdd V_g F_+(x,\xi) \cV_\g f(x,\xi) dxd\xi. 
\end{equation}
\end{theorem}
\begin{proof}
Let $f$ be a function in $\Sjord$ and fix $g,\gamma \in \cS(\rd)$ with $\lan \g,g \ran=1$. Given $\vp \in \cS(\rd)$ with $\vp(0)=1$, consider the following family of functionals indexed by $\veps \in (0,1)$:
\begin{equation}\label{eq:tveps}
    T_\veps(f) \coloneqq  \twpih^{-d/2} \ird e^{\hbid |x|^2}f(x)\vp(\varepsilon x) dx. 
\end{equation} Note that $T_\veps(f)$ is a well defined Lebesgue integral since $\Fp f \in L^\infty(\rd)$ and $\vp \in L^1(\rd)$. 

We prove below that $T_\veps(f)$ converges as $\veps \downarrow 0$ and provide a representation formula for the limit $I(f)$, namely
\begin{equation}
    I(f) \coloneqq \lan F_+, \bar{f} \ran_* = \frac{1}{\lan \g,g \ran} \irdd V_g F_+(x,\xi) \cV_\g f(x,\xi) dxd\xi.
\end{equation} 

First, we prove that the integral in \eqref{eq:tveps} coincides with the pairing between $f \vpe \in M^{\infty,1}(\rd)$ (we set $\vpe(x) \coloneqq \vp(\veps x)$ for brevity) and $F_+ \in M^{1,\infty}(\rd)$, namely it does coincide with the functional
\begin{equation}
    I_\veps(f) \coloneqq \lan F_+, \overline{f\vpe} \ran_* = \frac{1}{\lan \g,g \ran} \irdd V_g F_+(x,\xi) \cV_\g (f\vpe)(x,\xi) dxd\xi.
\end{equation}
Note that $I_\veps(f)$ is a well-defined Lebesgue integral in view of Lemmas \ref{lem:fresnel_function} and \ref{lem:boundH}, since 
\begin{equation}
    |V_g F_+(x,\xi)| |\cV_\g (f\vpe)(x,\xi)| \le C_N \lan x-\xi \ran^{-2N} H(\xi),  
\end{equation} that is a function in $L^1(\rdd)$ provided that $N$ is taken sufficiently large. 

The proof that $I_\veps(f) = T_\veps(f)$ then follows by the Parseval equality:
\begin{align}
    I_\veps(f) & = \frac{1}{\lan \g,g \ran} \irdd V_g \Fp(x,\xi) \cV_\g (f\vpe)(x,\xi) dxd\xi  \\ & = \frac{1}{\lan \g,g \ran} \ird \Big(
    \ird \cF [\Fp \cdot \overline{T_x g}](\xi) \overline{\cF[ \overline{f\vpe} \cdot \overline{T_x \g}](\xi)} d\xi \Big) dx \\
    & = \frac{1}{\lan \g,g \ran} \ird \Big( \ird \Fp(y) f(y)\vpe(y) \g(y-x)\overline{g(y-x)} dy \Big) dx \\
    & = \ird \Fp(y) f(y)\vpe(y) dy \\
    & = T_\veps(f). 
\end{align}

Finally, we claim that 
\begin{align} \lim_{\veps \downarrow 0} I_\veps(f) & = \lim_{\veps \downarrow 0} \frac{1}{\lan \g,g \ran} \irdd V_g F_+(x,\xi) \cV_\g (f\vpe)(x,\xi) dxd\xi \\ & = \frac{1}{\lan \g,g \ran} \irdd V_g F_+(x,\xi) \cV_\g f (x,\xi) dxd\xi \\ & = I(f). \end{align}
We argue by dominated convergence, resorting to Lemma \ref{lem:boundH} after noting that, since $f\vpe \to f$ in the sense of temperate distributions as $\veps \downarrow 0$, we have 
    \[ \cV_g(f\vpe)(x,\xi) \to \cV_gf(x,\xi), \quad \text{for all } (x,\xi) \in \rdd. \qedhere \] 
\end{proof}

The significance of this result is related to the fact that, as already anticipated, the Sj\"ostrand class $ M^{\infty,1}(\bR^d)$ is strictly larger than the Banach algebra $\cF\cM(\bR^d)$. We present below a concrete example of a smooth function with bounded derivatives (hence in $\Sjord$ --- see Section \ref{sec-sjo}) that does not arise as the Fourier transform of a measure. 

\begin{example}\label{example-embed}
Consider the function $f \colon \rd \to \bR$ given by $f(x)=\cos |x|$, which clearly belongs to $C^\infty_{\mathrm{b}}(\rd) \subset M^{\infty,1}(\rd)$. We claim that $f \notin \cF\cM(\rd)$ as long as $d>1$. In fact, we will consider only the case where $d=3$ here, leaving the details of the analogous proof of the other cases to the interested reader. 

We argue by contradiction, assuming that there exists $\mu \in \cM(\bR^3)$ such that $f=\cF \mu$, where the Fourier transform $\cF$ is here defined as in Section \ref{sec notation} with $\hbar=1$. Recall that the Fourier multiplier $m_t(D):=\cF^{-1} m_t \mathcal{F}$ with symbol $m_t(\xi)=\cos (t|\xi|)$, $t>0$, $\xi \in \bR^3$, corresponds to the evolution operator associated with the wave equation: 
\begin{equation}
\begin{cases}
    \partial_t^2 u(t,x)-\Delta_x u(t,x)= 0 \\
    u(0,x)= g(x), \qquad \partial_t u(0,x)=0, 
\end{cases} 
\end{equation} where $g \in \cS(\bR^3)$, that is
\begin{equation}
    u(t,x) = m_t(D) g(x) = (\cF^{-1}m_t \cF g)(x)=(\cF^{-1}m_t) \ast g(x), \qquad (t,x) \in \bR_+ \times \bR^3. 
\end{equation} In particular, since $m_1(\xi)=f(\xi)$, we have 
\begin{equation}
    u(1,x) = f(D)g(x)= \mu*g(x)= \int_{\bR^3} g(x-y)d\mu(y).
\end{equation} As a result, the functional $g \mapsto u(1,0)$ is continuous on the space $C_\mathrm{c}(\bR^3)$ of compactly supported continuous functions on $\bR^3$. On the other hand, the Kirchhoff formula for the solution of the wave equation via spherical means \cite[Section 2.4.1]{evans} yields
\begin{equation}
    u(t,x)= \partial_t \Big( t \dashint_{\partial B(x,t)} g(y) dS(y) \Big), 
\end{equation} where $\dashint_{\partial B(x,t)}$ denotes the mean over $\partial B(x,t)$ with respect to the surface measure $dS$. In particular, if $g$ is a radial function, i.e., $g(y) = \tilde g (|y|)$, we obtain
\begin{equation}
    u(1,0)=\partial_t (t \tilde{g}(t))|_{t=1}=\tilde g(1) + \tilde g'(1), 
\end{equation} which fails to be controlled by $\|g\|_{\infty}=\sup_{x \in \bR^3} |g(x)|$, in general.
\end{example}
\subsection{Fresnel integrability of distributions in $W^{\infty,1}$}\label{sec-int-fousjo}
The rest of this section is devoted to the study of the Fresnel integrability of \textit{distributions} belonging to the Wiener amalgam space $\fSjo(\rd) = \cF \Sjord $. Clearly, the Fresnel integral of a function $f \in \fSjo(\rd)$ cannot be defined as in \eqref{eq:fresnel int}, since \eqref{eq:tveps} generally fails to be a well defined Lebesgue integral in this case. 

The proof of Theorem \ref{thm:fr-int-sjo} suggests an alternative path to circumvent this obstruction. In particular, we first \textit{define} the Fresnel integral of $f \in \fSjord$ as follows: for all $\vp \in \cS(\rd)$ with $\vp(0)=1$ and $g,\g \in \cS(\rd)$ such that $\lan \g, g \ran \ne 0$, we set 
\begin{equation}\label{eq-virdstar} \vird^* e^{\hbid |x|^2}f(x) dx \coloneqq \lim_{\varepsilon \downarrow 0} \frac{1}{\lan \g,g \ran} \irdd V_g F_+(x,\xi) \cV_\g (f\vpe)(x,\xi) dxd\xi. 
\end{equation}
The following result shows that this alternative definition allows us to generalize Theorem \ref{thm:fr-int-sjo}.
\begin{theorem}\label{thm:fr-int-fsjo}
Every $f \in \fSjord$ is Fresnel integrable, and for every $g,\g \in \cS(\rd)$ such that $\lan \g, g \ran \ne 0$ we have
\begin{equation}
    \vird^* e^{\hbid |x|^2}f(x) dx = \frac{1}{\lan \g,g \ran} \irdd V_g F_+(x,\xi) \cV_\g f(x,\xi) dxd\xi. 
\end{equation}
\end{theorem}
\begin{proof} We argue by dominated convergence. First, note that the definition in \eqref{eq-virdstar} is well posed in the sense of Lebesgue integration, since by Lemmas \ref{lem:fresnel_function} and \ref{lem:boundH} we have
\begin{equation}
    |V_g F_+(x,\xi)| |\cV_\g (f\vpe)(x,\xi)| \le C_N \lan x-\xi \ran^{-2N} H(x),  
\end{equation} that is a function in $L^1(\rdd)$ provided that $N$ is sufficiently large. The latter condition and Lemma \ref{lem:boundH} are enough to obtain the claim, after noting that, since $f\vpe \to f$ in the sense of temperate distributions as $\veps \downarrow 0$, we have 
    \[ \cV_g(f\vpe)(x,\xi) \to \cV_gf(x,\xi), \quad \text{for all } (x,\xi) \in \rdd. \qedhere\] 
\end{proof}

\section{Towards the infinite dimensional extension} \label{sez-inf-dim-ext}
Let us  now  study the problem of extending Fresnel integration to the case where $\bR^d$ is replaced with an infinite dimensional space. We already anticipated in the Introduction that an effective approach relies upon the construction of a linear continuous functional $L \colon D(L)\to\bC$ for a suitable domain $D(L)$ of integrable functions. Clearly, this pathway works (and yields an integral with respect to a $\sigma$-additive measure) under the assumption of local compactness of the underlying topological space, which cannot be fulfilled in the case of an infinite dimensional Banach space, hence in this context the notion of linear continuous functional effectively extends the one of integral. 

We generalize below the construction in \cite{AlBr,AlHKMa,ELT} and enlarge the class of infinite dimensional Fresnel integrable functions by extending the results of Theorem \ref{thm:fr-int-sjo} to the case where $\bR^d$ is replaced by the set $\bR^\infty=\bR^\bN$ of real-valued sequences. To this end, we resort to the recent theory of projective system of functionals that has been recalled in Section \ref{sez-psf}. It is clear that $L(f)$ should reduce to a well-defined (finite dimensional) Fresnel integral in the case where the function $f$ depends explicitly only on a finite number of variables, namely if $f$ is a  {\em cylinder function} as defined in Section \ref{sez-cylinder function}.

\subsection{Fresnel integrals as projective system of functionals}

Let us consider now the particular case of Example \ref{example1}, where $A = \bN$, $E_J = \bR^{J}$, $J \in A$, and $E_A=\bR^\bN=\bR^\infty$ is the space of real-valued sequences.

For $m\leq n$, the symbols $\pi^n_m$ and $\pi_n$ denote respectively the projection maps $\pi_m^n \colon \bR^n\to \bR^m$ and $\pi_n \colon \bR^\infty\to \bR^n$. Recalling that $\hat E_n = \{f \colon \bR^n\to\bC\}$ and $\hat E_\infty = \{f \colon \bR^\infty \to\bC\}$, the symbols $\Ep_m^n$ and $\Ep_n$ similarly denote the extension maps $\Ep_m^n \colon\hat E_m \to \hat E_n$ and $\Ep_n \colon \hat E_n \to \hat E_\infty$ defined by
\begin{align*}
    \Ep_m^nf(x)&\coloneqq f(\pi^n_mx), \qquad f\in \hat E_m, \, x\in \bR^n,\\
    \Ep_nf({ x})&\coloneqq f(\pi_{ n}x), \qquad f\in \hat E_n, \, x\in \bR^\infty.
\end{align*}

Let us consider the family of mappings $L_n \colon M^{\infty,1}(\bR^n)\to \bC$, $n \in \bN$, defined by
\begin{align}
   D(L_n) = & M^{\infty,1}(\bR^n)\nonumber \\
   L_n(f) = & \widetilde{\int_{\bR^n}}e^{\hbid\|x\|^2}f(x)dx, \qquad f\in M^{\infty,1}(\bR^n). \label{eq-def-Ln}
\end{align}
It is clear that $L_n$ is a linear continuous functional in view of Theorem \ref{thm:fr-int-sjo} and the following inequality holds for every $g,\gamma \in \cS(\rn)$:
\begin{align}
   \left|\widetilde{\int_{\bR^n}}e^{\hbid |x|^2}f(x) dx \right|& =\left| \frac{1}{\lan \g,g \ran} \int_{\mathbb{R}^{2n}} V_g F_+(x,\xi) \cV_\g f(x,\xi) dxd\xi\right|  \nonumber\\
   &\leq \frac{1}{|\lan \g,g \ran|}\|F_+\|_{M_{(g)}^{1,\infty}}\|f\|_{M_{(\gamma)}^{\infty,1}}.\label{continuity:inequality}
\end{align}
It is also quite easy to check that the family $\{L_n,D(L_n)\}_{n\in \bN}$ is a projective system of functionals, since the requirements in Definition \ref{def-projectivesystemsfunctionals} are satisfied. 

A suitable choice of window functions paves the way for the extensions of the projective system  $\{L_n,D(L_n)\}_{n\in \bN}$. To be precise, fix a sequence $(q_j)_{j \in \bN}$ of positive real numbers and consider the family of windows $g_n\in \cS(\bR^n)$ of the form
\begin{equation}\label{window-g}
    g_n(x)=(2\pi\hbar)^{-n/2}e^{-\frac{1}{2\hbar}\langle x,Q_nx\rangle}, \quad x\in \bR^n,
\end{equation}
where $Q_n$ is a positive definite $n\times n$ diagonal matrix with strictly positive eigenvalues $q_1,\ldots, q_n$. With this choice and reference to notation introduced in Section \ref{sec gabor}, for $m<n$ and $f_m \in M^{\infty,1}(\bR^m)$, the STFT of the extended function $\Ep_m^n f_m$ is explicitly given by:
\begin{align*}
    V_{g_n}\Ep_m^n f_m(x,\xi)&=(2\pi\hbar)^{-n/2}\int_{\bR^n}e^{-\frac{i}{\hbar}\langle y,\xi\rangle}\Ep_m^n f_m(y){g_n}(y-x)dy\\
    &=\left((2\pi\hbar)^{-m/2}\int_{\bR^m}e^{-\frac{i}{\hbar}\sum_{j=1}^my_j\xi_j}f_m(x_1,\cdots,x_m)\frac{e^{-\frac{1}{2\hbar}\sum_{j=1}^mq_j(y_j-x_j)^2}}{(2\pi\hbar)^{m/2}}dy_1...dy_m\right)\\
    &\quad \times \left((2\pi\hbar)^{-(n-m)/2}\int_{\bR^{n-m}}e^{-\frac{i}{\hbar}\sum_{j=1}^{n-m}y_j\xi_j}\frac{e^{-\frac{1}{2\hbar}\sum_{j=m+1}^nq_j(y_j-x_j)^2}}{(2\pi\hbar)^{(n-m)/2}}dy_m...dy_{n-m}\right)\\
    &=V_{g_m}f_m(x_1,...,x_m,\xi_1,...,\xi_m)\frac{e^{-\frac{i}{\hbar}\sum_{j=1}^{n-m}x_j\xi_j}e^{-\frac{1}{2\hbar}\sum_{j=m+1}^nq_j^{-1}\xi_j^2}}{(2\pi\hbar)^{(n-m)/2}(\prod_{j=m+1}^nq_j)^{1/2}}, \qquad x,\xi \in \rn,
\end{align*}
therefore for all $m \le n$:
\begin{align*}
    \|\Ep_m^n f_m\|_{ M_{(g_n)}^{\infty,1}(\bR^n) }&=\int_{\bR^n}\sup_{x\in \bR^n}|V_{g_n}\Ep_m^n f_m(x,\xi)|d\xi\\
    &=\int_{\bR^m}\sup_{x\in \bR^m}|V_{g_m}f_m(x_1,...,x_m,\xi_1,...,\xi_m)|d\xi_1...d\xi_m\\
    &=\| f_m\|_{ M_{(g_m)}^{\infty,1}(\bR^m)}.
\end{align*}
We have thus obtained that the norm of any map $f_m\in M^{\infty,1}(\bR^m)$ coincides with the norm of its extensions $\Ep_m^n f_m$ for all $n\geq m$:
\begin{equation}\label{norm:consistency}
    \| f_m\|_\Mium = \|\Ep_m^n f_m\|_\Miun.
\end{equation}

The previous result allows us to define a norm on the set $\cC_0$ of cylinder functions  
\begin{equation}
    \cC_0\coloneqq\bigcup_n \Ep_n M^{\infty,1}(\bR^n).
\end{equation}
Indeed, given a cylinder function of the form $f=\Ep_nf_n$, for some $f_n\in M^{\infty,1}(\bR^n)$, we define its norm $\|f\|_\Miuinf $ by
\begin{equation}\label{def:norm:cylinder}
    \|f\|_\Miuinf \coloneqq \|f_n\|_\Miun, \qquad f=\Ep_nf_n\,.
\end{equation}
As a matter of fact, the norm $\|f\|_\Miuinf$ is unambiguously defined since it does not depend on the particular representation of the cylinder function $f$. The proof of this claim relies on a general argument that applies to a generic perfect inverse system, and goes as follows. Let $f$ be a cylinder function having two equivalent representations, such as $f=\Ep_mf_m=\Ep_nf_n$ for some $f_m\in E_m=M^{\infty,1}(\bR^m)$ and $f_n\in E_n=M^{\infty,1}(\bR^n)$. For any $N\geq m,n$ we thus have $f=\Ep_N\Ep_m^Nf_m=\Ep_N\Ep_m^Nf_m$, hence $\Ep_N\Ep_m^Nf_m=\Ep_N\Ep_n^Nf_n$ and the surjectivity of the projection $\pi_N$ implies $\Ep_m^Nf_m=\Ep_n^Nf_n$. The consistency property \eqref{norm:consistency} finally gives 
\[ \|f_n\|_\Miun =\|f_m\|_\Mium,\] showing that $\|f\|_\Miuinf $ is unambiguously defined.

Let us also stress that the operator norm of each functional $L_n$ is bounded by means of inequality \eqref{continuity:inequality}. In particular, for every window $\gamma \in \cS(\bR^n)$ we have 
\begin{align}
    \|L_n\|_{(g)}&\coloneqq \sup_{f\in M^{\infty,1}(\bR^n)\smo}\frac{|L_n(f)|}{\|f\|_\Miun}\\
    & = \sup_{f\in M^{\infty,1}(\bR^n)\smo}\frac{\left| \frac{1}{\lan g_n,\g \ran} \int_{\bR^{2n}} V_\gamma F_+(x,\xi) \cV_{g_n} f(x,\xi) dxd\xi\right| }{\|f\|_\Miun }\\
    &\leq \frac{1}{|\lan g_n,\g \ran|}\|F_+\|_{M_{(\gamma)}^{1,\infty}(\bR^n)}.\label{inequality:normL1}
\end{align}
The following result provides the exact value of the operator norm $\|L_n\|_{(g)}$. When the context is clear we omit the dependence on the window $g$ and write $\|L_n\|$ to lighten the notation.
\begin{theorem}\label{th:norm:Ln}
Consider, for any $n\in \bN$, the linear continuous functional $L_n \colon M^{\infty,1}(\bR^n)\to \bC$ defined in \eqref{eq-def-Ln}. We have
\begin{equation}
    \|L_n\|_{(g)}=\prod_{j=1}^n(q_j^2+1)^{1/4} = \inf_{\gamma\in \cS(\bR^n)}\frac{1}{|\lan g_n,\g \ran|}\|F_+\|_{M_{(\gamma)}^{1,\infty}(\bR^n)}.
\end{equation}
\end{theorem}
\begin{proof}
    Let us consider a parametric family of window functions $\gamma \in \cS(\bR^n)$ of the form
    \[\gamma (x)=e^{-\frac{\alpha -i}{2\hbar}\|x\|^2}\,, \qquad x\in \bR^n\]
    with real $\alpha >0$. By direct computation we get 
    \begin{equation}
        \|F_+\|_{M_{(\gamma)}^{1,\infty}(\bR^n)}=1\,.
    \end{equation}
    Similarly, we obtain
    \begin{equation}
        |\lan g_n,\g \ran|=\prod_{j=1}^n\left((\alpha +q_j)^2+1\right)^{-1/4},
    \end{equation}
    hence
    \begin{equation}
        \frac{\|F_+\|_{M^{1,\infty}_{(\gamma)}(\bR^n)}}{ |\lan g_n,\g \ran|}=\prod_{j=1}^n\left((\alpha +q_j)^2+1\right)^{1/4}.
    \end{equation}
    By virtue of inequality \eqref{inequality:normL1}, we infer that for any $\alpha>0$  the operator norm $\|L_n\|$ satisfies 
    \begin{equation}
        \|L_n\|\leq \prod_{j=1}^n\left((\alpha +q_j)^2+1\right)^{1/4},
    \end{equation}
    and taking the limit for $\alpha \downarrow 0$ eventually yields
    \begin{equation}\label{inequality:normL2}
        \|L_n\|\leq \prod_{j=1}^n(q_j^2+1)^{1/4}.
    \end{equation}
    Conversely, let us consider the family of Schwartz functions defined by $f_\varepsilon (x)=e^{-\frac{\varepsilon+i}{2\hbar}\|x\|^2}$, with $\varepsilon >0$. By direct computation we have
    \begin{equation}
        L(f_\veps)=(2\pi i\hbar )^{-n/2}\irn e^{-\frac{\veps}{2\hbar}\|x\|^2} dx=(i\varepsilon)^{-n/2}.
    \end{equation}
    Moreover, 
    \begin{equation}
        V_{g_n}f_\varepsilon (x,\xi)=\prod_{j=1}^n\frac{e^{-\frac{q_j}{2\hbar}x_j^2}e^{-\frac{(\xi_j+iq_jx_j)}{2\hbar(q_j+\varepsilon+i)^2}}}{\sqrt{2\pi\hbar(q_j+\varepsilon +i)}}\,,
    \end{equation}
    and 
    \begin{equation}
        \|f_\veps\|_\Miun=\prod_{j=1}^n((q_j+\varepsilon)^2+1)^{1/4}\left(\frac{1+\veps(q_j+\veps)}{\veps(q_j^2+2q_j\veps+1+\veps^2)}\right)^{1/2}\,,
    \end{equation}
    hence for any $\veps >0 $ we have
    \begin{equation}
         \|L_n\|\geq \frac{| L(f_\veps)|}{ \|f_\veps\|_{M^{\infty,1}_{(g)}(\bR^n)}}=\prod_{j=1}^n((q_j+\veps)^2+1)^{-1/4}\left(\frac{1+\veps(q_j+\veps)}{q_j^2+2q_j\veps+1+\veps^2}\right)^{-1/2},
    \end{equation}
    and in the limit regime $\veps\downarrow 0$ we obtain
    \begin{equation} \label{inequality:normL3}
        \|L_n\|\geq \prod_{j=1}^n (q_j^2+1)^{1/4}. 
    \end{equation}
    Combining \eqref{inequality:normL2} and \eqref{inequality:normL3} finally gives $ \|L_n\|=\prod_{j=1}^n (q_j^2+1)^{1/4}$. Combining this result with the argument leading to inequality \eqref{inequality:normL2} also proves the characterization in the claim. 
\end{proof}
As a consequence of this result, we note that a suitable choice of eigenvalues of the matrices $Q_n$ appearing in the definition of the windows $g_n$ allows one to ensure that the norms of the functionals $L_n$ are uniformly bounded with respect to $n$. 
\begin{corollary}\label{teo:bounded:operator}
With reference to \eqref{window-g}, if $ \sum_{n \in \bN} q_n^2 <\infty$ then $\displaystyle \sup_{n \in \bN}\|L_n\|<\infty$.
\end{corollary}
\begin{proof}
    By Theorem \ref{th:norm:Ln}, for any $n$  we have
   \begin{equation}
       \|L_n\|=\prod_{j=1}^n(q_j^2+1)^{1/4}=\exp \left( \frac{1}{4}\sum_{j=1}^n\log(q_j^2+1)\right). 
   \end{equation}
    Since $\log(1+x) \sim x$ as $x \downarrow 0$, convergence of the series $\sum_{n\in \bN}\log(1+q_n^2)$ is equivalent to convergence of the series $\sum_{n \in \bN} q_n^2$.
\end{proof}

\subsection{Interlude --- applications to the Schr\"odinger equation}\label{sec-schro}
Let us discuss here an interesting byproduct of Theorem \ref{th:norm:Ln}, namely the exact value of the $\Sjo(\rn) \to L^\infty(\rn)$ operator norm for the free particle Schr\"odinger evolution associated with the equation
\[
i\hbar\partial_t u=-\frac{\hbar^2}{2}\Delta u. 
\]
The solution $u(t,x)$, with $t \in \bR\smo$ and $x \in \rn$, such that $u(0,x)=f(x)$ has a standard integral representation:
\[
u(t,x)=e^{\frac{i\hbar t}{2}\Delta}f(x)=\frac{1}{(2\pi i \hbar t)^{n/2}}\int_{\mathbb{R}^n} e^{\frac{i}{\hbar }\frac{|y|^2}{2t}}f(x-y)\, dy.
\] 
Consider again the function $g_n\in \cS(\bR^n)$ given by
\begin{equation}
    g_n(x)=(2\pi\hbar)^{-n/2}e^{-\frac{1}{2\hbar}\langle x,Qx\rangle}, \quad x\in \bR^n
\end{equation}
where $Q$ is the $n\times n$ diagonal matrix with strictly positive eigenvalues $q_1,\ldots, q_n$. 
\begin{theorem}\label{thm-schro}
    We have 
    \[
\sup_{f\in M^{\infty,1}(\mathbb{R}^n)\setminus\{0\}}\frac{\|e^{\frac{i\hbar t}{2}\Delta}f\|_{L^\infty(\mathbb{R}^n)}}{\|f\|_{M^{\infty,1}_{(g_n)}(\mathbb{R}^n)}} = \prod_{j=1}^n(t^2q_j^2+1)^{1/4}.
    \]
\end{theorem}
 \begin{proof}
First of all we observe that, if $f\in M^{\infty,1}(\mathbb{R}^n)$ then $e^{\frac{i\hbar}{2}\Delta}f \in M^{\infty,1}(\rn)$ since the Schr\"odinger propagator is a Fourier multiplier with symbol in $W^{1,\infty}(\rn)$ (hence bounded on every modulation space --- see e.g., \cite[Proposition 3.6.9]{NT_book})), and thus a bounded continuous function on $\rn$. Therefore, $\|e^{\frac{i\hbar t}{2}\Delta}f\|_{L^\infty(\mathbb{R}^n)}= \sup_{x\in \mathbb{R}^n}|e^{\frac{i\hbar t}{2}\Delta}f(x)|$, and 
 \[
\sup_{f\in M^{\infty,1}(\mathbb{R}^n)\setminus\{0\}}\frac{\|e^{\frac{i\hbar t}{2}\Delta}f\|_{L^\infty(\mathbb{R}^n)}}{\|f\|_{M^{\infty,1}_{(g_n)}(\mathbb{R}^n)}}= \sup_{x\in\mathbb{R}^n}
\sup_{f\in M^{\infty,1}(\mathbb{R}^n)\setminus\{0\}}\frac{|e^{\frac{i\hbar t}{2}\Delta}f(x)|}{\|f\|_{M^{\infty,1}_{(g_n)}(\mathbb{R}^n)}}.
 \]
Assuming without loss of generality $t>0$, we can write 
\begin{equation}\label{eq schr}
e^{\frac{i\hbar t}{2}\Delta}f(x)= \widetilde{\irn}  e^{\frac{i}{\hbar }\frac{|y|^2}{2t}}f(x-y)\, dy,
\end{equation}
where in the Fresnel integral $\widetilde{\irn}$ (along with its normalization constant) we now have $\hbar t$ in place of $\hbar$. We can compute the supremum over $f$ by means of Theorem \ref{th:norm:Ln}, where the role of $\hbar$ is now played by $\hbar t$ and $Q$ is replaced by $tQ$ --- a similar adjustment is also needed in the short-time Fourier transform, which is implicitly involved in $\|f\|_{M^{\infty,1}_{(g_n)}(\mathbb{R}^n)}$. Combining Theorem \ref{th:norm:Ln} with the invariance of the norm $\|\cdot\|_{M^{\infty,1}_{(g_n)}(\mathbb{R}^n)}$ under translations and reflections, we infer that, for every $x\in\mathbb{R}^n$,
\[
\sup_{f\in M^{\infty,1}(\mathbb{R}^n)\setminus\{0\}}\frac{|e^{\frac{i\hbar t}{2}\Delta}f(x)|}{\|f\|_{M^{\infty,1}_{(g_n)}(\mathbb{R}^n)}}= \prod_{j=1}^n(t^2q_j^2+1)^{1/4},
\]
that is the claim.
  \end{proof}

\section{Infinite dimensional Fresnel integrals as projective systems of functionals --- topological approach}\label{sez-proj-ext-top}
The present section is devoted to the construction and the study of the properties of a projective extension $(L,D(L))$ (in the sense of Definition \ref{def-proj-syst}) of the projective system of functionals  $\{L_n,D(L_n)\}_{n\in \bN}$ defined in Section \ref{sez-inf-dim-ext}. In particular, the construction presented below relies on Theorem \ref{th:norm:Ln}, Corollary \ref{teo:bounded:operator} and a continuity argument.

Let us fix a sequence $(q_n)_{n \in \bN}$ of positive real numbers such that $\sum_n q_n^2 <\infty$, and consider the corresponding family of diagonal matrices $Q_n \colon \bR^n\to\bR^n$ with eigenvalues $q_1,\dots, q_n$. In fact, for each $n$ the operator $Q_n$ can be regarded as the restriction to $\bR^n$ of the  Hilbert-Schmidt operator $Q \colon \ell^2\to \ell^2$ defined on the Hilbert space $\ell^2(\bN)$ of square-integrable real-valued sequences $x = (x_n)\in \ell^2$ by $(Qx)_n\coloneqq q_nx_n$, $n \in \bN$. Consider now the projective system of functionals $\{L_n,D(L_n)\}_{n\in \bN}$ given by
\begin{equation}
    D(L_n)=\left(M^{\infty,1}(\bR^n), \| \cdot \|_\Miun\right), \qquad L_n(f)=\widetilde{\int_{\bR^n}}e^{\hbid |x|^2}f(x) dx,
\end{equation}
and its minimal extension $(\Lmin,D(\Lmin))$ defined as follows:
\begin{align}
    &D(\Lmin)=\cC_0 = \bigcup_n \Ep_n M^{\infty,1}(\bR^n),\label{domain:min:ex}\\
    &\Lmin f =L_n(f_n), \quad \text{where } f= \Ep_n f_n \text{ and } f_n\in M^{\infty,1}(\bR^n). \label{op:min:ex}
\end{align} 

We now construct and study a non-trivial extension $(L,D(L))$ of $(\Lmin,D(\Lmin))$. First of all, let us define the domain  $D(L)$ as the closure $D(L)\coloneqq \overline{\cC_0} $ of the set of cylinder functions in the norm $\| \cdot \|_\Miuinf$ defined in \eqref{def:norm:cylinder}. In other words, an element $f\in D(L) $ is associated with an equivalence class of Cauchy sequences of cylinder functions, say $(f_n)_{n \in \bN}$, and its norm $\|f\|_\Miuinf$ is given by the limit of the norms of the functions in the approximating sequence:
    \begin{equation}
        \|f\|_\Miuinf = \lim_{n\to \infty} \|f_n\|_\Miun.
    \end{equation}
The value of the limit on the right-hand side does not actually depend on the particular choice of the representative in the equivalence class of sequences associated with $f$. 

The image $L(f)$ of a function $f\in D(L)$  is defined, in the same spirit, via the limit
    \begin{equation}L(f)\coloneqq \lim_{n\to\infty} L(f_n), \label{def:op:cont:ex}
    \end{equation}
where $L(f_n)$ is given by \eqref{op:min:ex} for $f_n\in\cC_0$. 
The limit in \eqref{def:op:cont:ex} exists and is finite in view of Theorem \ref{teo:bounded:operator}, which accounts for the boundedness (hence the continuity) of the operator $L$ on the normed space $(D(L),\| \cdot \|_\Miuinf)$. The same result shows that the right-hand side of \eqref{def:op:cont:ex} does not depend on the choice of the representative $(f_n)_{n \in \bN}$ in the equivalence class of $f\in D(L)$. 

We have thus defined a linear continuous functional $(L,D(L))$ that extends $(\Lmin,D(\Lmin))$. Before proceeding with the analysis of its properties, it is worthwhile to examine an explicit example to better understand the construction. 
 
\begin{example}\label{examplef-nCauchy}
   A special instance of an element of $D(L)$, i.e., a Cauchy sequence $(f_n)_{n \in \bN} \subset \cC_0$ in  the norm $\| \cdot \|_\Miuinf $ can be constructed as follows:
   \begin{equation} f_n(x)\coloneqq \prod_{j=1}^nh_j(x_j)\, ,\qquad x\in \bR^\infty,\end{equation}
   where the functions $h_j \colon \bR\to \bC$, $j\in \bN$, are in $M^{\infty,1}(\bR)$ and $h_n \to 1$ in $M^{\infty,1}(\bR)$. To be concrete, let us consider the case where $h_j$ is defined by
   \begin{equation} h_j(x)=1+a_je_j(x), \qquad x\in \bR, \end{equation}
   where $e_j\in M^{\infty,1}(\bR)$ and $(a_n)_{n \in \bN}\subset \bR_+$ is a sequence of positive reals such that $a_n \downarrow 0$. Setting 
   \begin{equation} \tilde g_j(u)\coloneqq \frac{e^{-\frac{q_j}{2\hbar}u^2}}{\sqrt{2\pi\hbar}},\qquad u\in \bR, \end{equation}
   the norm of the cylinder function $f_n$, $n\in \bN$, is bounded as follows:
   \begin{align*}
       \|f_n\|_\Miuinf &=\prod_{j=1}^n \|h_j\|_{M_{(\tilde g_j)}^{\infty,1}(\bR)}\\
       &\leq \prod_{j=1}^n\Big(1+ a_j\|e_j\|_{M_{(\tilde g_j)}^{\infty,1}(\bR)}\Big).
   \end{align*}
   Therefore, a sufficient condition to ensure that $\sup_n \|f_n\|_\Miuinf<\infty$ is the convergence of the series $\sum_j a_j\|e_j\|_{M_{(\tilde g_j)}^{\infty,1}(\bR)}$ --- which holds, for instance, if $\sup_n\|e_n\|_{M_{(\tilde g_j)}^{\infty,1}(\bR)}<\infty$ and $\sum_na_n<\infty$. 
   
   Let us now compute $\|f_n-f_m\|_\Miuinf $, with $m<n$. To this aim, consider the sequence of functions $\tilde f_n \colon \bR^n\to \bC$ such that $f_n=\Ep_n\tilde f_n$. Then
   \[ \|f_n-f_m\|_\Miuinf =\|\tilde f_n-\Ep_m^n\tilde f_m\|_\Miun. \] More precisely, we have
   \begin{multline*}
       V_{g_n}\tilde f_n(x,\xi)-  V_{g_n}\Ep_m^n\tilde f_m(x,\xi)=V_{g_m}\tilde f_m(x_1,...,x_m,\xi_1,..., \xi_m)\\
       \times \left(\prod_{j=m+1}^nV_{\tilde g_j}h_j(x_j,\xi_j)-\prod_{j=m+1}^nV_{\tilde g_j}1(x_j,\xi_j)\right).
   \end{multline*}
   Recalling that $h_j(x)=1+a_je_j(x)$, we infer
   \begin{align}
       V_{g_n}\tilde f_n(x,\xi) -  V_{g_n}\Ep_m^n\tilde f_m(x,\xi)  & \begin{multlined}[t] =V_{g_m}\tilde f_m(x_1,...,x_m,\xi_1,..., \xi_m) \\
       \times \left(\prod_{j=m+1}^n(V_{\tilde g_j}1(x_j,\xi_j)+a_jV_{\tilde g_j}e_j(x_j,\xi_j))-\prod_{j=m+1}^nV_{\tilde g_j}1(x_j,\xi_j)\right) \end{multlined} \\
       & \begin{multlined}[t] =V_{g_m}\tilde f_m(x_1,...,x_m,\xi_1,..., \xi_m)\\ \hspace{-3.5cm}
       \times \sum_{k=1}^{n-m} \,\, \sum_{1\leq j_1<\dots <j_k\leq n-n} \,\, \prod_{l=1}^{k}a_{m+j_l}V_{\tilde g_{m+j_l}}e_{m+j_l}(x_{m+j_l},\xi_{m+j_l})\prod_{\substack{i=1 \\ i\neq j_1,\dots,j_k}}^{n-m}V_{\tilde g_{m+i}}1(x_{m+i},\xi_{m+i}). \end{multlined}
   \end{align}
   Noting that $\|1\|_{M_{(g)}^{\infty,1}(\bR)}=1$, we obtain
    \begin{align*}
    \|\tilde f_n-\Ep_m^n\tilde f_m\|_\Miun &\leq \|\tilde f_m\|_\Mium \sum_{k=1}^{n-m} \,\, \sum_{1\leq j_1<\dots <j_k\leq n-m}\,\, \prod_{l=1}^{k}a_{m+j_l}\|e_{m+j_l}\|_{M_{(\tilde g_{m+j_l})}^{\infty,1}(\bR)}\\
    &=\|\tilde f_m\|_\Mium\left(\prod_{j=m+1}^n(1+a_j\|e_j\|_{M_{(\tilde g_{j})}^{\infty,1}(\bR)})-1\right)\,,
       \end{align*}
       and the right-hand side vanishes as $m\to\infty$ if $\sum_ja_j\|e_j\|_{M_{(\tilde g_{j})}^{\infty,1}(\bR)}<+\infty$. 
       
By direct calculation, $L(f)$ is given by
\begin{equation}\label{OperatorLexamplef-nCauchy}
    L(f)=\lim_{n\to \infty}L_n(\tilde f_n)=\prod_{j\geq 1}(1+a_j L_1(e_j))\,.
\end{equation}
In the special case where $e_j(x)=e^{ik_jx}$ we have $\|e_j\|_{M_{(\tilde g_{j})}^{\infty,1}(\bR)}=1$, and  the following identities hold:
       \begin{equation*}
           \|f_n\|=\prod_{j=1}^n(1+a_j), \qquad L_n(f_n)=\prod_{j=1}^n(1+a_je^{-\frac{i\hbar}{2}k_j}).
       \end{equation*}
       Moreover, one can directly prove that the limit of the Cauchy sequence $(f_n)_{n \in \bN}$ is given by the map $f \colon \bR^\infty \to \bC$,  $f(x)=\prod_{j\geq 1}h_j(x_j)$. Indeed, for any $m$ the norm of the function $f-f_m$ coincides with the limit of the norms of an approximating sequence of cylinder functions, hence:
       \begin{align*}
           \|f-f_m\|_\Miuinf &=\lim_{\substack{n\to \infty \\ n\geq m}}\|f_n-f_m\|_\Miuinf\\
           &=   \|f_m\|_\Mium \lim_{\substack{n\to \infty \\ n\geq m}} \left(\prod_{j=m+1}^n(1+a_j\|e_j\|_{M_{(\tilde g_{j})}^{\infty,1}(\bR)})-1\right)\\
           &=   \|f_m\|_\Mium\left(\prod_{j>m}(1+a_j\|e_j\|_{M_{(\tilde g_{j})}^{\infty,1}(\bR)})-1\right)\\
           &\leq \prod_{j=1}^m\left(1+ a_j\|e_j\|_{M_{(\tilde g_{j})}^{\infty,1}(\bR)}\right)\left(\prod_{j>m}(1+a_j\|e_j\|_{M_{(\tilde g_{j})}^{\infty,1}(\bR)})-1\right)\,.
       \end{align*}
       Taking the limit for $m\to \infty$ makes the term $\prod_{j=1}^m\left(1+ a_j\|e_j\|_{M_{g_j}^{\infty,1}(\bR)}\right)$ to converge to a finite value, while $\lim_{m\to \infty}\left(\prod_{j>m}(1+a_j\|e_j\|)-1\right)$ vanishes.
   \end{example}
  \subsection{Characterization of the domain $D(L)$}
  We now focus on the description of the properties of the elements $f\in D(L)$ obtained as limit of Cauchy sequences of cylinder functions $(f_n)_{n\in\bN}$ in the norm $\| \cdot \|_\Miuinf$. We show that $f$ can be actually obtained as the \textit{pointwise} limit of the sequence of the cylinder functions. To this aim, we need to introduce a preliminary result relating the norms of restrictions on $M^{\infty,1}$. 

\begin{proposition}\label{Prop-comparison-norm}
    For $1\leq m<n$, consider the splitting $\bR^n=\bR^{m}\times \bR^{n-m}$, with coordinates $x=(x',x'')$, $x'\in\bR^m$, $x''\in\bR^{n-m}$.
Consider any window $g'\in \mathcal{S}(\bR^m)\setminus\{0\}$ and the window $g''\in \mathcal{S}(\bR^{n-m})$ given by
\begin{equation}
     g''(x'')=(2\pi\hbar)^{-(n-m)/2}e^{-\frac{1}{2\hbar}\langle x'',Q x''\rangle}
\end{equation}
where $Q$ is the $(n-m)\times (n-m)$ diagonal matrix with strictly positive eigenvalues $q_{ m+1},\ldots, q_{ n}$. 

Then, regarding $\bR^m$ as a subspace of $\bR^n$ via the inclusion $x'\mapsto (x',0)$, for every $f\in M^{\infty,1}(\bR^n)$ we have $f|_{\bR^{m}}\in M^{\infty,1}(\bR^m)$ with 
\begin{equation}\label{eq formula uno}
\|f|_{\bR^{m}}\|_{M^{\infty,1}_{(g')}(\bR^m)}\leq \|f\|_{M^{\infty,1}_{(g'\otimes g'')}(\bR^n)}.
\end{equation}
\end{proposition}
\begin{proof}
First, we observe that if $f\in M^{\infty,1}(\bR^n)$ then $f$ is continuous, and its restriction $f\mapsto f|_{\bR^m}$ is therefore well-defined pointwise. We also highlight that it suffices to prove \eqref{eq formula uno} for $f\in \cS(\bR^n)$. Indeed, if $f\in M^{\infty,1}(\rn)$ and $f_k\in\cS(\bR^n)$, with $f_k\to f$ in the sense of narrow convergence, then $f_k\to f$ pointwise and  $\|f_k\|_{M^{\infty,1}_{(g'\otimes g'')}}\to \|f\|_{M^{\infty,1}_{(g'\otimes g'')}}$ (by dominated convergence). Moreover, $f_k|_{\bR^m}$ is a Cauchy sequence in $M^{\infty,1}(\bR^m)$ in light of \eqref{eq formula uno} (applied to $f_n-f_m$), hence converges --- necessarily to $f|_{\bR^m}$ because convergence in $M^{\infty,1}$ implies pointwise convergence (see Section \ref{sec-sjo} for further details). Therefore, we assume $f\in\cS(\bR^n)$ hereinafter.

Let $\gamma\in\cS(\bR^{n-m})$, with $\langle \gamma, g''\rangle\not=0$. By Fubini's theorem and the Fourier inversion formula, we can represent the restriction $f|_{\bR^m}$ as 
\[
f(x',0)=\frac{(2\pi \hbar)^{-(n-m)/2}}{\langle \gamma, g''\rangle}\int_{\bR^{2(n-m)}} V_{g''}f(x';x'',\omega'')\gamma(-x'')\, dx''\,d\omega''.
\] 
Taking the short-time Fourier transform $V_{g'}$ and using Fubini's theorem twice, we obtain 
\[
V_{g'}f|_{\bR^m}(x',\omega')= \frac{(2\pi \hbar)^{-(n-m)/2}}{\langle \gamma, g''\rangle} \int_{\bR^{2n}} V_{g'\otimes g''} f(x',x'',\omega',\omega'')\gamma(-x'')\, dx'\, dx''\, d\omega'\, d\omega''. 
\]
It follows that 
\[
\|f|_{\bR^m}\|_{M^{\infty,1}_{(g')}(\bR^m)}\leq \frac{(2\pi \hbar)^{-(n-m)/2}\|\gamma\|_{L^1}}{|\langle \gamma, g''\rangle|}\|f\|_{M^{\infty,1}_{(g'\otimes g'')}(\bR^n)}.
\]
Consider now
\[\gamma(x'')=\gamma_\lambda(x'')\coloneqq  (2\pi\hbar\lambda^{-1})^{-(n-m)/2}e^{-\frac{1}{2\hbar}\lambda |x''|^2}, \quad\lambda>0.
\]
Explicit calculations show that 
\[
\frac{(2\pi \hbar)^{(n-m)/2}\|\gamma_\lambda\|_{L^1}}{|\langle \gamma_\lambda, g''\rangle|}=\prod_{j=1}^{n-m} \big(\frac{\lambda+q_j}{\lambda}\big)^{1/2},
\]
hence we infer
\[
\|f|_{\bR^m}\|_{M^{\infty,1}_{(g')}(\bR^m)}\leq \Big(\prod_{j=1}^{n-m} \frac{\lambda+q_j}{\lambda}\Big)^{1/2}\|f\|_{M^{\infty,1}_{(g'\otimes g'')}(\bR^n)}.
\]
Letting $\lambda\to+\infty$ then yields \eqref{eq formula uno}.  
\end{proof}
\begin{remark}
 Formula \eqref{eq formula uno} shows that the norm of the restriction operator, as a map $M^{\infty,1}(\bR^n)\to  M^{\infty,1}(\bR^m)$, with windows $g'\otimes g''$ and $g'$ respectively, is not larger than $1$. Taking $f=1$ shows that the norm is actually $1$. We recall that several trace theorems have been proved for modulation spaces, see for instance \cite{fei_11} and the references therein. Nevertheless, we stress that the previous result comes with the best constant in the bound, and this is the most relevant aspect for our purposes. 
\end{remark}
The case $m=0$ in Proposition \ref{Prop-comparison-norm} (more precisely, the same proof, with obvious adjustments), along with the translation invariance of the $M^{\infty,1}$ norm, implies the following result.
\begin{lemma}\label{lemma-in-norms}
    For the window $g_n$ defined in \eqref{window-g}, the following inequality holds for any $f\in M^{\infty,1}(\bR^n)$:
    \begin{equation}
        \|f\|_{L^\infty}\leq \|f\|_\Miun. 
    \end{equation}
\end{lemma}
   
   Let us consider now $f\in D(L)=\overline{\cC_0}$. By construction, $f$ is an equivalence class of Cauchy sequences of cylinder functions $(f_n)_{n\in\bN}$ with respect to the norm $\| \cdot \|_\Miuinf$. Given $k\in \bN$, let us consider the subset $P_k\bR^\infty$ of $\bR^\infty$ of real sequences $(x_n)_n$ such that $x_m=0$ for all $m>k$, and let  $f_m^{(k)} \colon \bR^k\to\bC$ denote the map related to the restriction of  $f_m$  to the finite-dimensional subspace $P_k\bR^\infty$, namely
\[ f_m^{(k)}(x_1,\ldots, x_k)\coloneqq f_m(x_1,\ldots, x_k,0,\ldots).\]

\begin{theorem} \label{thm-indep-rep} For all $k\in \bN$  the sequence ${f_n^{(k)}}$ is Cauchy in $M^{\infty,1}(\bR^k)$. Its limit, denoted by $f^{(k)}$, does not depend on the representative of $f$.
\end{theorem}
\begin{proof}
    Let us assume that, for any $n$, the cylinder function $f_n\in \cC_0$ is the extension of a map $\tilde f_n \colon \bR^{d_n}\to \bC$. Without loss of generality we can also assume that the sequence $(d_n)_{n \in \bN}$ is strictly increasing. Assuming $m<n$, the very definition of the norm $\|\ \cdot \|_\Miuinf$ implies
    \begin{equation}\label{equality-norm}
         \|f_n-f_m \|_\Miuinf=\|\tilde f_n-\Ep_{d_m}^{d_n}\tilde f_m\|_{M_{(g_{d_n})}^{\infty,1}(\bR^{d_n})}.
    \end{equation}
    Moreover, for $n,m$ sufficiently large, we  have $k<\min\{d(n),d(m)\}$, and Proposition \ref{Prop-comparison-norm} yields
\begin{equation}
    \| f_n^{(k)}-f_m^{(k)}\|_{M_{(g_k)}^{\infty,1}(\bR^k)}=\|(\tilde  f_n)_{|\bR^k} -(\Ep_{d_m}^{d_n}\tilde  f_m)_{|\bR^k} \|_{M_{(g_k)}^{\infty,1}(\bR^k)}\leq \|\tilde f_n-\Ep_{d_m}^{d_n}\tilde f_m\|_{M_{(g_{d_n})}^{\infty,1}(\bR^{d_n})}.
\end{equation}
The claim thus follows by \eqref{equality-norm} and the Cauchy property of the sequence $(f_n)_{n\in\bN}$. In particular, there exists a map $f^{(k)}\in M^{\infty,1}(\bR^k)$ such that 
\begin{equation}
    \lim_{n\to \infty}\|  f_n^{(k)} -f^{k} \|_{M_{(g_k)}^{\infty,1}(\bR^k)} =0.
\end{equation}
If $(f_n')_{n\in\bN}$ is a Cauchy sequence equivalent to $(f_n)$, arguing as above we obtain 
\begin{equation}
    \|  f_n^{(k)} -{f'_n}^{(k)} \|_{M_{(g_k)}^{\infty,1}(\bR^k)}\leq \|f_n-f_n'\|_\Miuinf,
\end{equation}
hence the restricted sequences $(f_n^{(k)})_{n\in\bN}$ and $({f'_n}^{(k)})_{n\in\bN}$ converge to the same limit.
\end{proof}
\begin{remark}\label{rem-point-fnk}
   It is clear from the proof and Lemma \ref{lemma-in-norms} that the sequence $(f_n^{(k)})_{n \in \bN}$  converges \textit{pointwise} to the limit function $f^{(k)}$.
\end{remark}
\begin{corollary}\label{corollary-values-f}
    For every $x\in \bR^\infty$, the limit $\lim_{n\to \infty}f_n(x)$ exists and is finite. Furthermore, the limit does not depend on the representative of the equivalence class of $f$.
\end{corollary}
\begin{proof}
    The result follows at once from the pointwise inequality
\begin{equation}
    |f_n(x)-f_m(x)|\leq \|f_n-f_m\|_\Miuinf.
\end{equation}
    In order to prove this bound, assume without loss of generality that $d_m\leq d_n$, and apply Lemma \ref{lemma-in-norms}:
    \begin{align*}
    |f_n(x)-f_m(x)|&=|\Ep_{d_n}\tilde f_n (x)-\Ep_{d_m}\tilde f_m (x)|\\
    &=|\tilde f_n (\pi_{d_n}x)-\tilde f_m (\pi_{d_m}x)| \\
    &=|(\tilde f_n -\Ep_{d_m}^{d_n}\tilde f_m)(\pi_{d_n}x)| \\
    &\leq \|\tilde f_n -\Ep_{d_m}^{d_n}\tilde f_m\|_{M_{(g_{d_n})}^{\infty,1}(\bR^{d_n})}\\
    &=\|f_n-f_m\|_\Miuinf.
    \end{align*}
Consider now another representative $(f_n')_{n\in\bN}$ of the equivalence class of $f$. Arguing as above, it is easy to see that for all $x\in \bR^\infty$ the sequences $(f_n(x))_{n\in\bN}$ and $(f'_n(x))_{n\in\bN}$ converge to the same limit. To be precise, assuming without loss of generality that $d(n)>d'(n)$, we have
\begin{align*}
    |f_n(x)-f'_n(x)|&=|\Ep_{d_n}\tilde f_n (x)-\Ep_{d'_n}\tilde f'_n (x)|\\
    &= |\tilde f_n (\pi_{d_n}x)-\tilde f'_n (\pi_{d'_n}x)|\\
    &\leq \|\tilde f_n -\Ep_{d'_n}^{d_n}\tilde f'_n\|_{M_{(g_{d_n})}^{\infty,1}(\bR^{d_n})}\\
     &=\|f_n-f_n'\|_\Miuinf. \qedhere
     \end{align*}
\end{proof}

As a consequence of Corollary \ref{corollary-values-f}, one can unambiguously associate the equivalence class $f$ of Cauchy sequences with a map $\bR^\infty\to \bC$, denoted again by $f$ with a slight abuse of notation, according to the rule
\begin{equation}
    f(x)\coloneqq \lim_{n\to\infty}f_n(x), \qquad x\in \bR^\infty. 
\end{equation}
In light of Remark \ref{rem-point-fnk}, the function $f^{(k)}$ obtained as the limit in $M^{\infty,1}(\bR^k)$ of the sequence $(f_n^{(k)})_{n\in\bN}$ coincides with the restriction $f|_{P_k\bR^\infty}$ of $f$ on the finite-dimensional subspaces $P_k\bR^\infty$:
\begin{equation}
    f^{(k)}(x_1,\dots,x_k)=f(x_1,\dots,x_k,0,\dots).
\end{equation}
Moreover, the following inequality holds:
\begin{equation}\label{inequality-n-fin-inf}
    \|f^{(k)}-f_n^{(k)}\|_{M_{(g_k)}^{\infty,1}(\bR^k)}\leq \|f_n-f\|_\Miuinf.
\end{equation}
    Indeed, the right-hand side of \eqref{inequality-n-fin-inf} is given (by definition) by
    \begin{equation}
        \|f_n-f\|_\Miuinf=\lim_{m\to\infty}\|f_n-f_m\|_\Miuinf.
    \end{equation}
    On the other hand, by Proposition \ref{Prop-comparison-norm}, for every $k\in \bN$ we have
    \begin{equation}
        \|f_n-f_m\|_\Miuinf\geq  \|f_n^{(k)}-f_m^{(k)}\|_{M_{(g_k)}^{\infty,1}(\bR^k)},
    \end{equation}
    and letting $m\to +\infty $ eventually yields \eqref{inequality-n-fin-inf}.

   \section{Infinite dimensional Fresnel integrals as projective systems of functionals --- sequential approach} \label{sez-proj-est-seq}
       
It is interesting to point out that even rather elementary functions $f \colon \bR^\infty \to\bR$ cannot be easily recaptured as the limit in the norm $\| \cdot \|_\Miuinf$ of quite natural cylinder functions, as shown in the following examples. 
           
\begin{example}\label{example-1}
Given $k\in \bR^\infty$, consider the map $f \colon \bR^\infty\to\bC$ given by 
\begin{equation}\label{function-example-1}
    f(x)\coloneqq \begin{cases}
    e^{\frac{i}{\hbar}k\cdot x}\, &\text{if }  k\cdot x\in \bR,\\
    0 & \text{otherwise.}
    \end{cases}
    \end{equation}          
Consider the sequence $(f_n)_{n \in \bN}$ of cylinder functions defined as 
    \begin{equation}\label{cyl:fun:fn}
    f_n(x)=e^{\hbi k_n\cdot x}, \qquad x\in \bR^\infty,
    \end{equation}
with $k_n=P_nk$. It is easy to realize that $f_n(x)$ explicitly depends only on the first $n$ components of the vector $x\in\bR^\infty$. Recall the choice of windows $g_n$ from \eqref{window-g}, and \eqref{def:norm:cylinder}. We claim that
\begin{equation}
    \|f_n-f_m\|_{M^{\infty,1}_{(g)}(\bR^\infty)}=2, \qquad \text{for any } m,n \in \bN \text{ such that } k_m \neq k_n.
\end{equation}
In order to show this equality, assume without loss of generality $m<n$ and denote by $\tilde f_n$ the map $\tilde f_n \colon \bR^n\to\bC$ such that $f_n=\Ep_n\tilde f_n$. We thus have
            \begin{equation}
                \|f_n-f_m\|_\Miuinf=\|\tilde f_n-\Ep_m^n\tilde f_m\|_\Miun.
            \end{equation} In particular, the following identity holds:
            \begin{equation}
                V_{g_n}\tilde f_n(x,\xi)=\frac{(2\pi i\hbar)^{-n/2}}{\sqrt{\det Q}}e^{\hbi x\cdot (k_n-\xi)}e^{-\frac{1}{2\hbar}\langle k_n-\xi,Q^{-1}(k_n-\xi)\rangle }. 
            \end{equation}
            Taking $n>m$ and setting $\delta\coloneqq k_n-k_m$, we have:
            \begin{multline*}
                V_{g_n}\tilde f_n(x,\xi)-V_{g_n}\Ep_m^n\tilde f_m(x,\xi)=\frac{(2\pi i\hbar)^{-n/2}}{\sqrt{\det Q}}e^{\hbi x\cdot (k_m-\xi)}e^{-\frac{1}{2\hbar}\langle k_m-\xi,Q^{-1}(k_m-\xi)\rangle }\\
                \times \left(e^{\hbi x\cdot \delta}e^{-\frac{1}{2\hbar}\langle \delta,Q^{-1}\delta\rangle } e^{-\frac{1}{\hbar}\langle\delta,Q^{-1}(k_m-\xi)\rangle }-1\right),
            \end{multline*}
            hence 
             \begin{equation}\label{difference-1}
                |V_{g_n}\tilde f_n(x,\xi)-V_{g_n}\Ep_m^n\tilde f_m(x,\xi)|=\frac{(2\pi \hbar)^{-n/2}}{\sqrt{\det Q}}e^{-\frac{1}{2\hbar}\langle k_m-\xi,Q^{-1}(k_m-\xi)\rangle } \sqrt{\rho(\delta)^2+1-2\cos(x\cdot \delta /\hbar)\rho(\delta)},
                \end{equation}
                where we set
                \begin{equation*}
                    \rho(\delta)\coloneqq e^{-\frac{1}{2\hbar}\langle \delta,Q^{-1}\delta\rangle } e^{-\frac{1}{\hbar}\langle\delta,Q^{-1}(k_m-\xi)\rangle }.
                \end{equation*}
        If $k_n=k_m$ for $m,n$ sufficiently large (i.e., if there exists an $M\in \bN$ such that $P_Mk=k$), then $\delta =0$ for all $n,m>M$ and the right-hand side of \eqref{difference-1} vanishes. On the other hand, if the vector $k\in \bR^\infty $ has an infinite number of non-vanishing components then $\delta \neq 0$ in general, and 
        \begin{align*}
                \sup_{x\in \bR^n}|V_{g_n}\tilde f_n(x,\xi)-V_{g_n}\Ep_m^n\tilde f_m(x,\xi)|&=\frac{(2\pi \hbar)^{-n/2}}{\sqrt{\det Q}}e^{-\frac{1}{2\hbar}\langle k_m-\xi,Q^{-1}(k_m-\xi)\rangle } \sqrt{\rho(\delta)^2+1+2\rho(\delta)}\\
                &=\frac{(2\pi \hbar)^{-n/2}}{\sqrt{\det Q}}e^{-\frac{1}{2\hbar}\langle k_m-\xi,Q^{-1}(k_m-\xi)\rangle } (\rho(\delta)+1)\\
                &=\frac{(2\pi \hbar)^{-n/2}}{\sqrt{\det Q}}\left(e^{-\frac{1}{2\hbar}\langle k_n-\xi,Q^{-1}(k_n-\xi)\rangle }+e^{-\frac{1}{2\hbar}\langle k_m-\xi,Q^{-1}(k_m-\xi)\rangle }\right),
                \end{align*}
eventually leading to 
\begin{equation}
    \int \sup_{x\in \bR^n}|V_{g_n}\tilde f_n(x,\xi)-V_{g_n}\Ep_m^n\tilde f_m(x,\xi)|d\xi =2.
\end{equation}
Therefore, $(f_n)_{n \in \bN}$ is not a Cauchy sequence. It is thus not surprising that the sequence also fails to converge in norm $\| \cdot \|_{\cF\cM}$ in the space of the Fourier transform of complex measures. In particular, for $f_n$ as in \eqref{cyl:fun:fn}, we have 
            $$f_n(x)=\int e^{iy\cdot x}\delta_{k_n/\hbar}(y).$$
            Therefore, for $k_m\neq k_n$ we immediately get
            \begin{equation}
                \|f_n-f_m\|_{\cF\cM}=|\delta_{k_n/\hbar}-\delta_{k_m/\hbar}|=2,
            \end{equation}
            and the sequence does not converge unless there exists some $M\in \bN$ such that $k_n=k_m$ for all $m,n>M$, i.e., $k$ has a finite number of non-vanishing components --- that is, the map $f$ is just a cylinder function.
        \end{example}
   \begin{example}\label{example-2}
       Another interesting example, similar to the previous one, is related to the cylindrical approximation of the (non-cylinder) Gaussian function $f\colon \bR^\infty\to\bC$ given by 
       \begin{equation}\label{function-example-2}
             f(x)\coloneqq \begin{cases}
                 e^{-\frac{1}{2\hbar}\sum _{n} r_nx_n^2} & \text{if } \sum _{n} r_nx_n^2<\infty,\\
                 0 &\text{otherwise},
             \end{cases}
       \end{equation} where $(r_n)_{n \in \bN}$ is a sequence of positive real numbers with infinitely many non-vanishing terms. Let us consider the sequence $(f_n)_{n \in \bN}$ of cylinder functions defined as 
            \begin{equation}\label{cyl:fun:fnGaussian}
                f_n(x)=e^{-\frac{1}{2\hbar}\sum _{k=1}^nr_kx_k^2}, \qquad x\in \bR^\infty.
            \end{equation}
            We claim that the sequence $(f_n)_{n \in \bN}$ is not a Cauchy sequence with respect to the norm $\|\cdot \|_{M^{\infty, 1}(\bR^\infty)}$. This can be seen after introducing the sequence of functions $\tilde f_n \colon \bR^n\to \bC$ such that $f_n=\Ep_n\tilde f_n$, for which we have (again for windows as in \eqref{window-g})
            \begin{equation*}
                V_{g_n}\tilde f_n(x,\xi)=(2\pi \hbar)^{-n/2}\prod_{j=1}^n\frac{1}{\sqrt{r_j+q_j}}e^{-\frac{\xi_j^2}{2\hbar(r_j+q_j)}}e^{-i\frac{q_jx_j\xi_j}{\hbar(r_j+q_j)}}e^{-\frac{r_jq_jx_j^2}{2\hbar(r_j+q_j)}}, 
            \end{equation*}
            and 
            \begin{equation*}
                V_{g_n}\Ep_m^n\tilde f_m(x,\xi)=(2\pi \hbar)^{-n/2}\prod_{j=1}^m\frac{1}{\sqrt{r_j+q_j}}e^{-\frac{\xi_j^2}{2\hbar(r_j+q_j)}}e^{-i\frac{q_jx_j\xi_j}{\hbar(r_j+q_j)}}e^{-\frac{r_jq_jx_j^2}{2\hbar(r_j+q_j)}}\prod_{j=m+1}^n\frac{1}{\sqrt q_j}e^{-\frac{i\xi_jx_j}{\hbar}}e^{-\frac{\xi_j^2}{2\hbar q_j}}.
            \end{equation*}
We thus infer
            \begin{multline*}  | V_{g_n}\tilde f_n(x,\xi)-  V_{g_n}\Ep_m^n\tilde f_m(x,\xi)|=|V_{g_m}\tilde f_m(x_1,...,x_m,\xi_1,...,\xi_m)|\\ \times(2\pi \hbar)^{-(n-m)/2}\left|\prod_{j=m+1}^n\frac{1}{\sqrt{r_j+q_j}}e^{-\frac{\xi_j^2}{2\hbar(r_j+q_j)}}e^{-i\frac{q_jx_j\xi_j}{\hbar(r_j+q_j)}}e^{-\frac{r_jq_jx_j^2}{2\hbar(r_j+q_j)}}-\prod_{j=m+1}^n\frac{1}{\sqrt q_j}e^{-\frac{i\xi_jx_j}{\hbar}}e^{-\frac{\xi_j^2}{2\hbar q_j}}\right|\,.
            \end{multline*}
           The last term on the right-hand side can be recast as 
             \begin{multline*}
             \left|\prod_{j=m+1}^n\frac{1}{\sqrt{r_j+q_j}}e^{-\frac{\xi_j^2}{2\hbar(r_j+q_j)}}e^{-i\frac{q_jx_j\xi_j}{\hbar(r_j+q_j)}}e^{-\frac{r_jq_jx_j^2}{2\hbar(r_j+q_j)}}-\prod_{j=m+1}^n\frac{1}{\sqrt q_j}e^{-\frac{i\xi_jx_j}{\hbar}}e^{-\frac{\xi_j^2}{2\hbar q_j}}\right|\\
             =\sqrt{a^2+b^2-2ab\cos\left(\sum_{j=m+1}^n\frac{q_jx_j\xi_j}{\hbar(r_j+q_j)}-\sum_{j=m+1}^n\frac{\xi_jx_j}{\hbar}\right)}\, ,
             \end{multline*}
             where we set $a\coloneqq \prod_{j=m+1}^n\frac{1}{\sqrt{r_j+q_j}}e^{-\frac{\xi_j^2}{2\hbar(r_j+q_j)}}e^{-\frac{r_jq_jx_j^2}{2\hbar(r_j+q_j)}}$ and $b \coloneqq \prod_{j=m+1}^n\frac{1}{\sqrt q_j}e^{-\frac{\xi_j^2}{2\hbar q_j}}$.
           We have
        \begin{equation}
            \lim_{x\to\infty}\sqrt{a^2+b^2-2ab\cos\left(\sum_{j=m+1}^n\frac{q_jx_j\xi_j}{\hbar(r_j+q_j)}-\sum_{j=m+1}^n\frac{\xi_jx_j}{\hbar}\right)}=b, 
        \end{equation}
            hence
             \begin{multline*}
               \sup_{x \in \rn} | V_{g_n}\tilde f_n(x,\xi)-  V_{g_n}\Ep_m^n\tilde f_m(x,\xi)|\geq  \sup_{x \in \rn} |V_{g_m}f_m(x_1,...,x_m,\xi_1,...,\xi_m)|\\ \times (2\pi \hbar)^{-(n-m)/2}\prod_{j=m+1}^n\frac{1}{\sqrt q_j}e^{-\frac{\xi_j^2}{2\hbar q_j}}\, .
               \end{multline*}
               To conclude, we have
               \[ \irn \sup_{x \in \rn} | V_{g_n}\tilde f_n(x,\xi)-  V_{g_n}\Ep_m^n\tilde f_m(x,\xi)|d\xi \geq 1,\]
               which shows that $(f_n)_{n \in \bN}$ fails to be a Cauchy sequence.
   \end{example}
  

The examples above show the difficulties related to the integration of rather simple non-cylinder functions. Indeed, the domain $D(L)=\bar \cC_0$ of the functional constructed in Section \ref{sez-proj-ext-top} appears to be too small to include a sufficiently large class of functions. Besides, under suitable assumptions on the vector $k\in \bR^\infty$ or the sequence $(r_n)_{n \in \bN} \subset\bR^+$, the functions \eqref{function-example-1} and \eqref{function-example-2} studied in Examples \ref{example-1} and \ref{example-2} belong to the Fresnel algebra $\cF(\ell^2(\bN))$ and are thus integrable according to the construction described in \cite{AlHKMa}.

In order to handle a larger class of integrable functions, so to further extend the domain of the functional, we will consider an alternative construction leading to the definition of a new extension $(L', D(L'))$ of $(L_{\min},D(L_{\min}))$.

Given $f \colon \bR^\infty \to \bC $, for all $n\in \bN$ let $f^{(n)} \colon \bR^n\to \bC$ be the  function defined as 
\begin{equation}\label{def-fn}
    f^{(n)}(x_1,\dots x_n)\coloneqq f(x_1,\dots x_n,0,0,...), \qquad (x_1,\dots x_n)\in \bR^n.
\end{equation}
Let us define the functional $(L',D(L'))$ as follows:
\begin{itemize}
    \item The domain $D(L')$ is the linear subspace of maps $f \colon \bR^\infty \to \bC$ such that the following conditions hold:
    \begin{enumerate}
        \item The map $ f^{(n)}$ belongs to $M^{\infty,1}(\bR^n)$ for every $n \in \bN$.
        \item The limit $\lim_{n\to\infty}L_n(f^{(n)})$ exists and is finite.
    \end{enumerate}
    \item If $f$ belongs to the domain $D(L')$ as defined above, we set $L'(f) \coloneqq \lim_{n\to \infty} L_n(f^{(n)})$.
\end{itemize}

It is straightforward to verify that the (generally non-cylinder) function \eqref{function-example-1} with $k\in \ell^2(\bN)$ belongs to $D(L)$ and $L(f)=e^{-\frac{i}{2\hbar}k^2}$.
In a similar fashion, assuming $\sum_nr_n<\infty$ ensures that the map \eqref{function-example-2} of Example \ref{example-2} satisfies $f\in D(L')$ and $L'(f)=\prod_n(1+ir_n)^{-1/2}$. 

More generally, it is easy to prove that the class of integrable functions in this context includes the algebra of Fourier transforms of complex measures on $\ell^2(\bN)$ (see \cite{AlBr,ELT,mazz_book}). Indeed, let $f \colon \bR^\infty\to \bC$ be of the form 
\begin{equation}
    f(x)\coloneqq \begin{cases}
        \int_{\ell^2}e^{i\langle x,y\rangle}d\mu(y)\qquad & x\in \ell^2(\bN),\\
        h(x) &x\in \bR^\infty\setminus \ell^2(\bN),
    \end{cases}
\end{equation}
where $\mu $ is a complex Borel measure on $\ell^2(\bN)$ and $h\colon \bR^\infty\setminus \ell^2(\bN)$ is an arbitrary function. In this case, each function $f^{(n)} \colon \bR^n\to\bC$ of the sequence defined in \eqref{def-fn} belongs to $\cF\cM(\bR^n)$ and can be represented as
\begin{equation}
    f^{(n)}(x)=\int_{\ell^2} e^{i\langle x,y\rangle}d\mu_n(y),\qquad x\in \bR^n,
\end{equation}
where $\mu_n$ is the complex Borel measure on $\bR^n$ obtained as the pushforward measure of $\mu$ under the projection map $\pi_n \colon \ell^2(\bN)\to\bR^n$. In view of \eqref{Parseval-type-rd}, the Fresnel integral of $f^{(n)}$
can be computed in terms of the Parseval type equality:
\begin{align*}
    L_n(f^{(n)})&=\int_{\bR^n }e^{-\frac{i\hbar}{2}|x|^2}d\mu_n (x)\\
    &=\int_{\ell^2 }e^{-\frac{i\hbar}{2}|\pi_n x|^2}d\mu (x).
\end{align*}
One can now take the limit of both sides for $n\to \infty$, and by dominated convergence theorem we obtain
$$L'(f)=\int_{\ell^2 }e^{-\frac{i\hbar}{2}|x|^2}d\mu (x).$$

{\begin{remark}
    Direct inspection of the definition of $(L',D(L'))$, as well as of the examples above, show that the image $L'(f)$ of $f\in D(L')$ depends only on the values that the map $f \colon \bR^{\infty}\to \bC$ attains on the set $c_{00}$ of finitely supported real sequences, namely $(x_n)_{n \in \bN}$ such that $x_n=0$ for all but finitely many $n$.
\end{remark}}

We are now ready to prove that $L'$ is an extension of the functional $L$ introduced in Section \ref{sez-proj-ext-top}, hence generalizing the topological construction described in Section \ref{sez-proj-ext-top}.
\begin{theorem} \label{thm-L'-ext}
For any $f\in D(L)$ the following holds:
\begin{equation} f\in D(L'), \qquad L'(f)=L(f). \end{equation}
\end{theorem}
\begin{proof}
    Consider $f\in D(L)$, that is an equivalence class of Cauchy sequences (in the norm $\|\cdot \|_\Miuinf$) of cylinder functions $(f_n)_{n\in\bN}$. Each cylinder map $f_n \colon \bR^\infty\to \bC$ can be represented as the extension of a function $\tilde f_n \colon \bR^{d_n}\to \bC$ in such a way that $f_n=\Ep_{d_n} \tilde f_n$. Without loss of generality, by exploiting identity \eqref{rel-Ep} we can assume that the sequence $(d_n)_{n\in\bN}$ is strictly increasing. In fact, it suffices to consider the case where $d_n=n$ --- otherwise, one can construct a sequence $(f'_n)_{n\in\bN}$, equivalent to $(f_n)_{n\in\bN}$ and satisfying $d'_n=n$, by setting
    \begin{equation}
        f'_m \coloneqq \begin{cases}
            f_n & (m=d_n)\\
            \Ep_{d_n}^mf_n     &  (d_n <m<d_{n+1}).
        \end{cases}
    \end{equation}
    
     Assume then $d_n=n$ from now on. For any $m\in \bN$, consider the restriction map $f^{(m)} \colon \bR^m\to \bC$ defined by \eqref{def-fn}. As discussed in Section \ref{sez-proj-ext-top}, $f^{(m)}$ coincides with the limit in $M^{\infty,1}(\bR^m)$ of the sequence of restricted cylinder functions  $f_n^{(m)}$. Hence, due to the closure of  $M^{\infty,1}(\bR^m)$, we get $f^{m}\in M^{\infty,1}(\bR^m)$ and we only need to show that the limits $\lim_{n\to\infty}L(f_n)$ and $\lim_{n\to\infty}L_n(f^{(n)})$ do coincide. In order to prove this, by the uniform boundedness of the family of operators $\{L_n\}$ (Theorem \ref{teo:bounded:operator}), it is enough to show that the sequences of cylinder functions $(f_n)_n$ and $(\Ep_m f^{(m)})_m$ are equivalent. Indeed, for $n>m$ we have:
\begin{align*}
     \|f_n-\Ep_m f^{(m)}\|_\Miuinf &=&& \|f_n-f|_{P_m\bR^\infty}\|_\Miuinf \\
    &=&&\| f_n-f_m+f_m-f_n|_{P_m\bR^\infty}+f_n|_{P_m\bR^\infty}-f|_{P_m\bR^\infty}   \|_\Miuinf \\
    & \leq && \| f_n-f_m  \|_\Miuinf +\|f_m-f_n|_{P_m\bR^\infty}\|_\Miuinf \\ 
    & && +\|f_n|_{P_m\bR^\infty}-f|_{P_m\bR^\infty}   \|_\Miuinf \\ 
    &=&& \| f_n-f_m  \|_\Miuinf+\|\tilde f_m-f_n^{(m)}\|_\Mium \\ 
    & && +\|f_n^{(m)}-f^{(m)}\|_\Mium.
\end{align*}
   In the last line, the first term $ \| f_n-f_m  \|_\Miuinf $ converges to 0 for $m,n\to \infty $ since $(f_n)_{n\in\bN}$ is a Cauchy sequence by assumption. Similarly, the third term $\|f_n^{(m)}-f^{(m)}\|_\Mium$ has the same behaviour by virtue of the inequality \eqref{inequality-n-fin-inf}. Concerning the second term $\|\tilde f_m-f_n^{(m)}\|_\Mium$, it can be rephrased as
   \begin{equation*}
       \|\tilde f_m-f_n^{(m)}\|_\Mium =\|(\Ep_m^n\tilde f_m)_{|\bR^m}-(\tilde f_n)_{|\bR^m}\|_\Mium,
   \end{equation*}
     and by Proposition \ref{Prop-comparison-norm} it is bounded by
     \begin{equation*}
     =\|(\Ep_m^n\tilde f_m)_{|\bR^m}-(\tilde f_n)_{|\bR^m}\|_\Mium\leq \|\Ep_m^n\tilde f_m-\tilde f_n\|_\Miun = \|f_m-f_n\|_\Miuinf, 
      \end{equation*}
      hence it vanishes as well for $n,m\to\infty$.
    \end{proof}

A concrete instance of this general result is provided by the function $f$ studied in Example \ref{examplef-nCauchy}. Since $f(x)=\prod_{j\geq 1}(1+a_je_j(x_j))$, the functions $\tilde f_n \colon \bR^n\to \bC$ are given by $\tilde f_n (x_1,\dots x_n)\coloneqq \prod_{j= 1}^n(1+a_je_j(x_j))\prod_{j> n}(1+a_je_j(0))$, where the convergence of the infinite product $\prod_{j> n}(1+a_je_j(0))$ follows from the assumption of the convergence of $\prod_{j> n}(1+a_j\|e_j\|_{M^{\infty,1}_{(g)}(\bR)})$ and the inequality $\|e_j\|_\infty\leq C_g\|e_j\|_{M^{\infty,1}_{(g)}(\bR)}$ (see Lemma \ref{lemma-in-norms}), where $C_g>0$ is a suitable constant that depends on the window $g$. Computing $ L_n(\tilde f_n)$ yields
\begin{equation}
     L_n(\tilde f_n)=\prod_{j> n}(1+a_j\|e_j\|_{M^{\infty,1}_{(g)}(\bR)})\prod_{j=1}^n(1+a_jL_1(e_j)),
\end{equation}
and since $\lim_{n\to \infty}\prod_{j> n}(1+a_j\|e_j\|_{M^{\infty,1}_{(g)}(\bR)})=1 $ we finally obtain
\begin{equation}
    L'(F)=\prod_{j\geq 1}(1+a_jL_1(e_j)),
\end{equation}
which coincides with \eqref{OperatorLexamplef-nCauchy}.

\subsection{A new class of integrable functions}

Let us now consider a concrete example of a map $f\in D(L')$ which does not belong to the class $\cF (\ell^2(\bN))$ of Fourier transforms of measures on $\ell^2(\bN)$, showing then that $D(L')$ extends even the Fresnel class $\cF \cM(\ell^2(\bN))$.

{ Let $h \in \Sjo(\bR)$ and $k=(k_n)_{n\in \bN} \in \ell^2(\bN)$ be a sequence of real numbers. Consider the  function $f \colon \bR^\infty \to \bC$ defined by
\begin{equation}
    f(x) \coloneqq \begin{cases}
        h(k\cdot x) & x\in \ell^2(\bN)\\
        0 & x\in \bR^\infty \setminus \ell^2(\bN)
    \end{cases}
\end{equation}
By construction, $f$ is not a cylinder function unless $k\in c_{00}$. Moreover, in the interesting case where $h\in \Sjo(\bR)\setminus \cF\cM(\bR)$ (see for instance Example \ref{example-embed}), the map $f$ does not belong to the Albeverio--H{\o}egh-Krohn class $\cF\cM(\ell^2(\bN))$ and thus provides a non trivial example of a new Fresnel integrable function.

We now prove that $f\in D(L')$. To this aim, for $N \in \bN$ set \begin{equation}
    \bk_N \coloneqq \pi_N k = (k_1, \ldots, k_N) \in \bR^N, \qquad \alpha_N(x) \coloneqq \pi_N k \cdot x = \sum_{j=1}^N k_j x_j, \quad x \in \bR^N, 
\end{equation} and consider the sequence of functions $(f^{n})_{n\in \bN}$, $f^{(n)} \colon \rn \to \bC$, defined by \eqref{def-fn}. These can be equivalently represented as  follows: 
\begin{equation} \label{f-n-fin} \qquad f^{(n)}(x) \coloneqq h(\alpha_n(x))\,,\qquad x\in \bR^n\,. 
\end{equation}
We now investigate the behavior of the sequence of Fresnel integrals 
\begin{equation}
    L_n(f^{(n)}) = \widetilde{\irn}e^{\hbid |x|^2}f^{(n)}(x)dx. 
\end{equation}
To this end, it is convenient to introduce an inversion formula, allowing to compute the values attained by a map in $f \in \Sjord$ in terms of the STFT of its Fourier transform $\hat f$.
}
Given $\omega \in \rd$, let $\psi_\omega \colon \rd \to \bC$ be the generalized pure tone at frequency $\omega$, namely $\psi_\omega(y) \coloneqq \twph^{-d/2} e^{\hbi \omega \cdot y}$. A straightforward computation shows that \[ V_g \psi_\omega (x,\xi) = \twph^{-d/2} e^{-\hbi (\xi-\omega) \cdot x} \overline{\hat{g}(\omega-\xi)}.\]

\begin{lemma}\label{lem:rep_infou}
    For every $f \in \Sjord$ and $g,\gamma \in \cS(\rd)$ such that $\lan g,\gamma \ran \ne 0$, we have the representation formula
    \begin{align} \label{eq:f_rep_invfou}
        f(t) & = \frac{1}{\lan \gamma,g \ran} \irdd V_g \wh{f}(x,\xi)  \overline{V_\g \psi_{-t}(x,\xi)} dxd\xi \\
        & = \frac{1}{\lan \gamma,g \ran} \irdd V_g \wh{f}(x,\xi) \cV_\g \psi_{t}(x,\xi) dxd\xi
        , \qquad t \in \rd. 
    \end{align}
\end{lemma}

\begin{proof}
    First, let us note that the integral is well defined: 
    \begin{equation}
        |V_g \wh{f}(x,\xi) \overline{V_\g \psi_{-t}(x,\xi)}| = \twph^{-d/2} |\hat{\g}(-\xi-t)| |V_g \wh{f}(x,\xi)| \in L^1(\rdd),
    \end{equation} since $|V_g \hat{f}(x,\xi)| \in L^1_x(L^\infty_\xi)$ by assumption. 

    The representation formula \eqref{eq:f_rep_invfou} holds if $f \in \cS(\rd)$ as a consequence of the Parseval theorem. Indeed,
    \begin{align*}
    \frac{1}{\lan \gamma,g \ran} \irdd V_g \wh{f}(x,\xi) \cV_\g \psi_{t}(x,\xi) dxd\xi & =  \frac{1}{\lan \gamma,g \ran} \irdd V_g \wh{f}(x,\xi) \overline{V_\g \psi_{-t}(x,\xi)} dxd\xi    \\
    & = \frac{1}{\lan \gamma,g \ran} \irdd \cF [ \wh{f} \cdot \overline{T_xg}](\xi) \overline{ \cF [\psi_{-t} \overline{T_x \gamma}](\xi)} dxd\xi \\ 
    & = \frac{1}{\lan \gamma,g \ran} \irdd \wh{f} \cdot \overline{T_xg}(\xi) \overline{ \psi_{-t}} T_x \gamma(\xi) dxd\xi \\
    & =\twph^{-d/2}\ird \wh{f} (\xi) e^{\hbi t\cdot \xi} d\xi \\
    & = f(t).
    \end{align*}
    In general, let $(f_n)_n$ be a sequence in $\cS(\rd)$ such that $f_n \to f$ in the narrow convergence sense. Then, in view of the pointwise bound $|V_g \wh{f_n}(x,\xi)|\le H(x)$ for an integrable function $H \in L^1(\rd)$ that does not depend on $n$, we argue by dominated convergence:
\begin{align}
    \irdd V_g \wh{f}(x,\xi) \cV_\g \psi_{t}(x,\xi) dxd\xi & =\irdd \Big( \lim_{n \to \infty} V_g \wh{f_n}(x,\xi) \Big) \cV_{\g} \psi_{t}(x,\xi)  dxd\xi  \\ & = \lim_{n\to \infty} \irdd V_g \wh{f_n}(x,\xi) \cV_{\g} \psi_{t}(x,\xi) dxd\xi 
    \\ & = \lan \gamma,g \ran \lim_{n\to \infty} f_n(t) \\
    & = \lan \gamma,g \ran f(t),
\end{align} where in the last step we used the fact that if $f_n \to f$ narrowly in $\Sjo$ then $f_n \to f$ pointwise --- in fact, uniformly over compact subsets of $\rd$: given a compact subset $K \subset \rd$ and $\Psi \in \cS(\rd)$ such that $\Psi=1$ on $K$, then
    \begin{equation}
        \| (f-f_n)\Psi \|_{\infty}  \le \| V_\Psi (f-f_n)(0,\cdot) \|_{L^1} \to 0. \qedhere
    \end{equation}
\end{proof}

Coming back to the computation of the Fresnel integrals $L_n(f^{(n)})$, with $f^{(n)}$ given by \eqref{f-n-fin},  using \eqref{eq:f_rep_invfou} with $g=\gamma$ and $\|g\|_{L^2}=1$ (for instance $\gau(y)=(\pi \hbar)^{-d/4} e^{-\frac{1}{2\hbar}|y|^2}$), we have
\begin{align*}
L_n(f^{(n)}) & = \widetilde{\irn}e^{\hbid |y|^2}f^{(n)}(y)dy \\
& = \lim_{\veps \downarrow 0} \, \twpih^{-n/2} \irn e^{\hbid |y|^2} h(\alpha_n( y)) \vp(\veps y) dy \\ 
& = \lim_{\veps \downarrow 0} \, \twpih^{-n/2} \irn e^{\hbid |y|^2} \Big( \int_{\bR^2} V_g \wh{h}(x,\xi) \overline{V_g \psi_{-\alpha_n(y)}(x,\xi)} dxd\xi \Big) \vp(\veps y) dy.
\end{align*}

It is straightforward to note that the function \[
F_{\veps}(y,x,\xi) \coloneqq e^{\hbid |y|^2} V_g \wh{h}(x,\xi) V_g \psi_{-\alpha_n(y)}(x,\xi) \vp(\veps y)
\]
belongs to $L^1(\rn \times \bR \times \bR)$ for every fixed $\veps >0$, hence by Fubini's theorem we have
\begin{equation}
    L_n(f^{(n)}) = \lim_{\veps \downarrow 0} \, \twpih^{-n/2} \int_{\bR^2} \Big( \irn e^{\hbid |y|^2} \overline{V_g \psi_{-\alpha_n(y)} (x,\xi)} \vp(\veps y) dy  \Big) V_g \wh{h}(x,\xi) dxd\xi. 
\end{equation}
By explicit computation (see Appendix \ref{appendix-formulaFresenel})
we get 
\begin{multline}
   \twpih^{-n/2}  \irn e^{\hbid |y|^2} {\overline{V_g \psi_{-\alpha_n(y)} (x,\xi)}} \vp(\veps y) dy  \\ =e^{\hbi x \xi}\twph^{-1}\int _{\bR^n\times  \bR}e^{-\frac{i}{2\hbar}|\lambda \bk_n +x \bk_n+\veps\hbar w|^2}\hat\varphi (w)e^{\frac{i}{\hbar}\lambda\xi} g(\lambda) dw d\lambda. \label{rep-FresnelVgpsi}
\end{multline}
By dominated convergence, exploiting the fact that $\varphi(0)=\int_{\bR^n}\hat\varphi (w)dw=1$,  we finally obtain: 
\begin{align}
     L_n(f^{(n)}) &= \twph^{-1} \int_{\bR^2} \left(\int_\bR e^{-\frac{i}{2\hbar}|\bk_n|^2(\lambda+x)^2}e^{\frac{i}{\hbar}(\lambda+x)\xi} {g(\lambda)} d\lambda \right)V_g \wh{h}(x,\xi) dxd\xi\\
     &= \twph^{-1} \int_{\bR^2} \left(\int_\bR e^{-\frac{i}{2\hbar}|\bk_n|^2y^2}e^{\frac{i}{\hbar}y\xi} {g(y-x)}dy\right)V_g \wh{h}(x,\xi) dxd\xi \\
     & = \twph^{-1} \int_{\bR^2} \overline{\left(\int_\bR e^{\frac{i}{2\hbar}|\bk_n|^2y^2}e^{-\frac{i}{\hbar}y\xi} \overline{g(y-x)}dy\right)} V_g \wh{h}(x,\xi) dxd\xi \\
     & = \int_{\bR^2} \overline{V_g(\Fp \circ \| \pi_n k\|)(x,\xi)} V_g \wh{h}(x,\xi) dxd\xi.
\end{align} 
To sum up, in light of \eqref{eq dual star}, we have the explicit formula
\begin{align}
    L_n(f^{(n)}) & = \widetilde{\irn}e^{\hbid |x|^2}f^{(n)}(x)dx \\ 
    &  =  \int_{\bR^2} V_g \wh{h}(x,\xi) \overline{V_{g} (\Fp \circ \|\pi_n k\|)(x,\xi)} dxd\xi  \\
    & =  \lan \wh{h},\Fp \circ \|\pi_n k\| \ran_*.
\end{align}
We are now concerned with the limit $\lim_{n\to \infty} L_n(f^{(n)})$, and we claim that
\begin{align}   \label{FinalResult}
   L'(f) & = \lim_{n\to \infty} L_n(f^{(n)}) \\ & = \int_{\bR^2} V_g \wh{h}(x,\xi) \overline{V_g (\Fp \circ \|k\|_{\ell^2})(x,\xi)}dxd\xi   \\ 
   &  = \lan \wh{h},\Fp \circ \|k\|_{\ell^2} \ran_*,
\end{align}
cf. \eqref{eq dual star}. Let us highlight that if $n$ is such that $\pi_n k = 0$ then
\begin{align*}
    L_n(f^{(n)}) & = \int_{\bR^2} V_g \wh{h}(x,\xi) \overline{V_g 1 (x,\xi)}  dxd\xi \\
    & = \lan \wh{h}, 1 \ran_* = f(0). 
\end{align*} 
Actually, the claim in \eqref{FinalResult} is a consequence of the following general result in the case where $H(x,\xi)=V_g \wh{h}(x,\xi)$ --- the proof being postponed to Appendix \ref{app-proof_Tn}. 
\begin{proposition}\label{prop-Tn}
In the setting introduced above, consider the family of linear functionals 
\begin{equation}
    T_n \colon L^1_x(L^\infty_\xi)(\bR^2) \to \bC, \qquad T_n(H) \coloneqq \int_{\bR^2} H(x,\xi) \overline{V_g (\Fp \circ \|\pi_n k\|)(x,\xi)} dxd\xi, 
\end{equation} and similarly set 
\begin{equation}
    T \colon L^1_x(L^\infty_\xi)(\bR^2) \to \bC, \qquad T(H) \coloneqq \int_{\bR^2} H(x,\xi) \overline{V_g (\Fp \circ \|k\|)(x,\xi)} dxd\xi.
\end{equation} Then: 
\begin{enumerate}
    \item The functionals $T_n$ are uniformly bounded with respect to $n$. 
    \item There exists a dense subspace $X \subset L^1_{x}(L^\infty_\xi)(\bR^2)$ such that $T_n(H)\to T(H)$ for all $H \in X$. 
\end{enumerate} As a result, we have $T_n(H)\to T(H)$ for all $H \in L^1_x(L^\infty_\xi)(\bR^2)$. 
\end{proposition}

\section*{Acknowledgements}
The authors are members of Gruppo Nazionale per l’Analisi Matematica, la Probabilità e le loro Applicazioni (GNAMPA) --- Istituto Nazionale di Alta Matematica (INdAM). F.~N.~ is a Fellow of the Accademia delle Scienze di Torino and a member of the Societ\`a Italiana di Scienze e Tecnologie Quantistiche (SISTEQ).

The present research has been developed as part of the activities of the GNAMPA-INdAM 2022 project ``Analisi armonica e stocastica in problemi di quantizzazione e integrazione funzionale'', award number (CUP): E55F22000270001. We gratefully acknowledge financial support from GNAMPA-INdAM.

Part of this research has been carried out at CIRM-FBK (Centro Internazionale per la Ricerca Matematica --- Fondazione Bruno Kessler, Trento), in the context of the ``Research in Pairs 2023'' program. S.\ M.\ and S.\ I.\ T.\ are grateful to CIRM-FBK for the financial support and the excellent facilities. 

The authors report there are no competing interests to declare. No dataset has been analysed or generated in connection with this note. 

\appendix
\section{Proof of formula {\eqref{rep-FresnelVgpsi}}}\label{appendix-formulaFresenel}

Let $g \in \cS(\bR)$, $\varphi \in \cS(\bR^n)$, $v\in \bR^n$  and $f\in \cF\cM(\bR^n)$, i.e., of the form $f(x)=\int_{\bR^n}e^{iu\cdot y }d\mu(u)$ for some complex Borel measure $\mu$ on $\bR^n$. Then, for any $\veps >0$, $\hbar >0$,
\begin{equation}\label{Formula1AppendixA}
    (2\pi i \hbar)^{-n/2}\int_{\bR^n} e^{\frac{i}{2\hbar}|y|^2}e^{iv\cdot y}f(y)\varphi (\veps y)dy=\int _{\bR^n\times \bR^n}e^{-\frac{i\hbar}{2}|u+v+\veps w|^2}\hat\varphi (w)dwd\mu(u),
\end{equation}
where $\varphi(x)=\int _{\bR^n} e^{ix\cdot u}\hat \varphi (u)du$, $x \in \rn$. Indeed:
\begin{equation} (2\pi i \hbar)^{-n/2}\int_{\bR^n} e^{\frac{i}{2\hbar}|y|^2}e^{iv\cdot y}f(y)\varphi (\veps y)dy =(2\pi i \hbar)^{-n/2}\int_{\bR^n} e^{\frac{i}{2\hbar}|y|^2}e^{iv\cdot y}\varphi (\veps y)\left(\int_{\bR^n}e^{iu\cdot y }d\mu(u)\right)dy. 
\end{equation}
 By Fubini's theorem, the latter is equal to 
 \begin{align*}
     \int_{\bR^n}\left(\int_{\bR^n} \frac{e^{\frac{i}{2\hbar}|y|^2}}{(2\pi i \hbar)^{n/2}} e^{i(v+u)\cdot y}\varphi (\veps y)dy \right)  d\mu(u) &=\int_{\bR^n}\left(\int_{\bR^n} e^{-\frac{i\hbar}{2}|w+u+v|^2} \veps ^{-n}\hat\varphi (w/\veps )dw \right)  d\mu(u)\\
     &=\int_{\bR^n}\left(\int_{\bR^n} e^{-\frac{i\hbar}{2}|\veps w+u+v|^2} \hat\varphi (w )dw \right)  d\mu(u). 
 \end{align*}  
Let us now consider the integral
 \begin{equation}
     I_{\veps,x,\xi}\coloneqq \twpih^{-n/2}  \irn e^{\hbid |y|^2} \overline{V_g \psi_{-\alpha_n(y)} (x,\xi)} \vp(\veps y) dy, \quad x,\xi \in \bR,
 \end{equation}
 where 
  \[ \overline{V_g \psi_{-\alpha_n(y)} (x,\xi)} = \twph^{-1/2} e^{\hbi x(\xi+\alpha_n(y)) } \hat{g}(-\alpha_n(y)-\xi),\]
  with $\alpha_n(y) \coloneqq \bk_n\cdot y$, $\bk_n \coloneqq \pi_n k = (k_1,\ldots,k_n)$. The integral $I_{\veps,x,\xi}$ can be recast as
  \begin{equation*}
      I_{\veps,x,\xi}=e^{\hbi x\xi}(2\pi i \hbar)^{-n/2}\int_{\bR^n} e^{\frac{i}{2\hbar}|y|^2}e^{iv\cdot y}f(y)\varphi (\veps y)dy,
  \end{equation*}
  where $v=x \bk_n/\hbar$ and $f$ is the map given by 
  \begin{align*}
      f(y)&=\twph^{-1/2}\hat{g}(-\alpha_n(y)-\xi)\\
      &=\twph^{-1} \int_\bR e^{-\frac{i}{\hbar}u(-\alpha_n(y)-\xi)}g(u)du\\
      &=\twph^{-1}\int_\bR e^{\frac{i}{\hbar}u\bk_n\cdot y }e^{\frac{i}{\hbar}u\xi}g(u) du\\
       &=\twph^{-1}\int_\bR \int_{\bR^n}e^{iw\cdot y }  \delta_{u\bk_n/\hbar} (w)e^{\frac{i}{\hbar}u\xi} g(u) du.
  \end{align*} 
  In particular, since $f\in \cF\cM(\bR^n)$, we have the representation $f(y)=\int_{\bR^n}e^{iw\cdot y }d\mu(w)$ with $d\mu=\twph^{-1}\int_\bR e^{\frac{i}{\hbar}u\xi}{g(u)} \delta_{u\bk_n/\hbar}du$. Then, by \eqref{Formula1AppendixA} we have 
  \begin{align*}
       I_{\veps,x,\xi}=& e^{\hbi x\xi} \int _{\bR^n\times \bR^n}e^{-\frac{i\hbar}{2}|u +x \bk_n/\hbar+\veps w|^2}\hat\varphi (w)dwd\mu(u)\\
       =&e^{\hbi x\xi}\twph^{-1}\int _{\bR^n\times \bR^n\times \bR}e^{-\frac{i\hbar}{2}|u + x\bk_n/\hbar+\veps w|^2}\hat\varphi (w)e^{\frac{i}{\hbar}\lambda\xi} {g(\lambda)} dw\delta_{\lambda\bk_n/\hbar}(u) d\lambda\\
       =&e^{\hbi x\xi}\twph^{-1}\int _{\bR^n\times  \bR}e^{-\frac{i\hbar}{2}|\lambda\bk_n/\hbar + x\bk_n/\hbar+\veps w|^2}\hat\varphi (w)e^{\frac{i}{\hbar}\lambda\xi} {g(\lambda)} dw d\lambda\\
        =&e^{\hbi x\xi}\twph^{-1}\int _{\bR^n\times  \bR}e^{-\frac{i}{2\hbar}|\lambda\bk_n + x\bk_n+\veps\hbar w|^2}\hat\varphi (w)e^{\frac{i}{\hbar}\lambda\xi} {g(\lambda)} dw d\lambda.
  \end{align*} 

\section{Proof of Proposition \ref{prop-Tn}}\label{app-proof_Tn}

\textbf{Step 1.} \textit{Uniform boundedness of the family $\{T_n\}$. }

Assuming $k \ne 0$, let thus $N$ be the smallest integer such that $\| \pi_N k \| \ne 0$. It follows by direct computation that 
\begin{equation}
    V_g (f\circ \lambda) (x,\xi) = |\lambda|^{-1} V_{g \circ \lambda^{-1}} f(\lambda x, \lambda^{-1}\xi). 
\end{equation} 
In particular, arguing as in the proof of Lemma \ref{lem:fresnel_function}, if $0<A \le \lambda \le B$ we have for all $m \in \bN$: 
\begin{align}
    |V_g (\Fp \circ \lambda)(x,\xi)| & = \lambda^{-1} |V_{g \circ \lambda^{-1}} \Fp(\lambda x, \lambda^{-1}\xi)| \\ 
    & \lesssim_{m,A} \lan \lambda x + \lambda^{-1}\xi \ran^{-2m}.
\end{align}  Moreover, after noticing that $\lambda^2 x$ belongs to the interval $[-B^2|x|, B^2|x|]$, we have
\begin{align}
    \lan \lambda x + \lambda^{-1} \xi \ran^{-2m} & = (1+ \lambda^{-2}|\xi + \lambda^2 x|^2)^{-m} \\ 
    & \le (1+ B^{-2}|\xi + \lambda^2 x|^2)^{-m} \\ 
    & \le \Phi_{m,B}(x,\xi), 
\end{align} where we introduced the function
\begin{equation}
    \Phi_{m,B}(x,\xi) \coloneqq \begin{cases}
        1 & (|\xi| \le B^2 |x|) \\ (1+B^{-2} \min\{|\xi - B^2 x|^2,|\xi+B^2x|^2\})^{-m} & (|\xi|>B^2|x|).
    \end{cases}
\end{equation} For future reference we emphasize that a straightforward tail bound shows that, for $m>1$,
\begin{equation}
    \int_{\bR} \Phi_{m,B}(x,\xi) d\xi \lesssim \lan x \ran.  \label{eq-phi-bound}
\end{equation} To sum up, we have obtained 
\begin{equation}
    |V_g(\Fp \circ \lambda)(x,\xi)|\lesssim_{m,A,B} \Phi_{m,B}(x,\xi).
\end{equation}

This is precisely the case under our attention: setting $\lambda_n = \| \pi_n k\|$, we have $0<|k_N| \le \lambda_n \le \|k\|_{\ell^2}$ for all $n \ge N$, hence for all $m \in \bN$
\begin{align*}
     |V_g (\Fp \circ \| \pi_n k\|)(x,\xi)| & \lesssim_{m,A} \lan \lambda_n x - \lambda_n^{-1}\xi \ran^{-2m} \\ 
     & \lesssim (1+ B^{-2}|\xi + \lambda_n^2 x|^2)^{-m} \numberthis \label{eq-Vgfmpink-bound-m} \\ 
     & \lesssim_{m,A,B} \Phi_{m,B}(x,\xi), \numberthis \label{eq-Vgfmpink-bound-phi}
\end{align*} where we set $A=|k_N|$ and $B=\|k\|_{\ell^2}$ --- in particular, the implicit constants do not depend on $n$. 

As a result, for all $H \in L^1_x(L^\infty_\xi)(\bR^2)$ and $m>1$, by \eqref{eq-Vgfmpink-bound-m} we have 
\begin{align*} 
|T_n(H)| & \le  \int_{\bR^2} |H(x,\xi) \overline{V_g (\Fp \circ \|\pi_n k\|)(x,\xi)} | dx d\xi \\ 
& \lesssim \int_{\bR} \sup_{\xi \in \bR} |H(x,\xi)| \Big( \int_{\bR} (1+ B^{-2}|\xi + \lambda_n^2 x|^2)^{-m} d\xi \Big) dx  \\
& \le C \| H \|_{L^1_x(L^\infty_\xi)},
\end{align*} for a constant $C>0$ that does not depend on $n$ and $H$. We thus conclude that the family of operators $(T_n)$ is uniformly bounded. 

\noindent \textbf{Step 2.} \textit{Dense subspaces $X_q$ of $L^1_x(L^\infty_\xi)$.} 

Given $q > 0$, consider the space
\begin{equation}
    X_q \coloneqq \{ H \colon \bR^2 \to \bC : \lan x \ran^q H(x,\xi) \in L^1_x(L^\infty_\xi) (\bR^2)\}. 
\end{equation}
It is clear from the very definition that $X_q$ is a subspace of $L^1_x(L^\infty_\xi) (\bR^2)$, with continuous embedding. We now prove that $X_q$ is a dense subset in $L^1_x(L^\infty_\xi) (\bR^2)$. 

Let $\varphi$ be the function
\begin{equation}
    \varphi(t) \coloneqq \begin{cases}
        1 & (|t| \le 1) \\ e^{-|t - 1|^2} & (|t|>1),
    \end{cases} \quad t \in \bR,
\end{equation} and consider the associated sequence $(\varphi_n)_{n \in \bN}$ obtained by $\varphi_n(t) \coloneqq \varphi(t/n)$. 

Let $H \in L^1_x(L^\infty_\xi) (\bR^2)$ and set $H_n(x,\xi) \coloneqq H(x,\xi) \varphi_n(x)$. It is easy to realize that $H_n \in X_q$ for every $n \in \bN$ and $q>0$, since
\begin{equation}
    \int_{\bR} \lan x \ran^q \varphi(x/n) \big(\sup_{\xi \in \bR} H(x,\xi)\big) dx < \infty. 
\end{equation}
Moreover, by dominated convergence we have
\begin{equation}
    \lim_{n\to \infty} \| H-H_n \|_{L^1_x(L^\infty_\xi)} = \lim_{n\to \infty} \int_{\bR} \big(\sup_{\xi \in \bR} |H(x,\xi)|\big) (1-\varphi_n(x)) dx = 0,
\end{equation} hence proving the density of $X_q$ as claimed.

\noindent \textbf{Step 3.} \textit{Convergence $T_n \to T$ on $X_q$.} 

Given $q\ge 1$, we prove now that $T_n(H)\to T(H)$ for every $H \in X_q$. To this aim, by dominated convergence we infer 
\begin{align}
    \lim_{n\to \infty} |T_n(H)-T(H)| & = \lim_{n\to \infty} \int_{\bR^2} |V_g (\Fp \circ \|\pi_n k\|)(x,\xi)-V_g (\Fp \circ \| k \|_{\ell^2})(x,\xi)| |H(x,\xi)| dxd\xi \\
    & = \int_{\bR^2} \lim_{n\to \infty} |V_g (\Fp \circ \|\pi_n k\|)(x,\xi)-V_g (\Fp \circ \| k \|_{\ell^2})(x,\xi)| |H(x,\xi)| dxd\xi \\
    & = 0.
\end{align} Indeed, $\lim_{n\to \infty} V_g (\Fp \circ \|\pi_n k\|)(x,\xi) = V_g (\Fp \circ \| k \|_{\ell^2})(x,\xi)$ for all $(x,\xi) \in \bR^2$. Moreover, by effect of \eqref{eq-Vgfmpink-bound-phi} we have
\begin{align}
    & |V_g (\Fp \circ \|\pi_n k\|)(x,\xi)-V_g (\Fp \circ \| k \|_{\ell^2} )(x,\xi)| |H(x,\xi)| \\ \le & \Big(|V_g (\Fp \circ \|\pi_n k\|)(x,\xi)|+|V_g (\Fp \circ \| k \|_{\ell^2} )(x,\xi)|\Big) \Big( \sup_{\xi \in \bR} |H(x,\xi)|\Big) \\
    \lesssim & \Phi_{m,B}(x,\xi)\Big( \sup_{\xi \in \bR} |H(x,\xi)|\Big) \\
    \eqqcolon & G(x,\xi).
\end{align} By \eqref{eq-phi-bound} and given that $H\in X_q$, we have
\begin{align}
    \|G\|_{L^1} & \le \int_{\bR^2} \Phi_{m,B}(x,\xi)\Big( \sup_{\xi \in \bR} |H(x,\xi)|\Big) dx d\xi \\
    & \le \int_{\bR} \Big( \sup_{\xi \in \bR} |H(x,\xi)|\Big)  \Big(\int_{\bR} \Phi_{m,B}(x,\xi) d\xi\Big) dx \\
    & \lesssim \int_{\bR} \lan x \ran \Big( \sup_{\xi \in \bR} |H(x,\xi)|\Big)  dx < \infty. \qedhere
\end{align}


\begin{thebibliography}{99}

 
 \bibitem{AlBr}
Albeverio, Sergio; Brze\'niak, Zdzis\l{}aw. Finite-dimensional approximation approach to oscillatory integrals and stationary phase in infinite dimensions. \textit{J. Funct. Anal.} \textbf{113} (1993), no. 1, 177--244.
\bibitem{AlCaMa}
Albeverio, S.; Cangiotti, N.; Mazzucchi, S. A rigorous mathematical construction of Feynman path integrals for the Schr\"odinger equation with magnetic field. \textit{Comm. Math. Phys.} \textbf{377} (2020), no. 2, 1461--1503. 

\bibitem{AlHKMa}
Albeverio, Sergio A.; H\o{}egh-Krohn, Raphael J.; Mazzucchi, Sonia. \textit{Mathematical Theory of Feynman Path Integrals. An Introduction.} Second edition. Lecture Notes in Mathematics, 523. Springer-Verlag, Berlin, 2008. 

\bibitem{AlMa2005}
Albeverio, Sergio; Mazzucchi, Sonia. Generalized Fresnel integrals. \textit{Bull. Sci. Math.} \textbf{129} (2005), no. 1, 1--23. 

\bibitem{AlMa2}
Albeverio, Sergio; Mazzucchi, Sonia. Feynman path integrals for polynomially growing potentials. \textit{J. Funct. Anal.} \textbf{221} (2005), no. 1, 83--121. 

 \bibitem{AlMa2016} Albeverio, Sergio; Mazzucchi, Sonia. A unified approach to infinite-dimensional integration. \textit{Reviews in Mathematical Physics} \textbf{28} (2016), no. 2, 1650005.

 \bibitem{benyi_unimod} B\'enyi, \'Arp\'ad; Gr\"ochenig, Karlheinz; Okoudjou, Kasso A.; Rogers, Luke G. Unimodular Fourier multipliers for modulation spaces. \textit{J. Funct. Anal.} \textbf{246} (2007), no. 2, 366--384. 

 \bibitem{Bau} 
Bauer, Heinz. \textit{Probability Theory.} Translated from the fourth (1991) German edition by Robert B. Burckel and revised by the author. De Gruyter Studies in Mathematics, 23. Walter de Gruyter \& Co., Berlin, 1996. 

\bibitem{BBR} Boggiatto, Paolo; Buzano, Ernesto; Rodino, Luigi. \textit{Global Hypoellipticity and Spectral Theory.} Mathematical Research, 92. Akademie Verlag, Berlin, 1996. 
 
 \bibitem{Boc} Bochner, Salomon. \textit{Harmonic Analysis and the Theory of Probability.} Courier Corporation, 2005.

 \bibitem{Cam}
Cameron, Robert H. A family of integrals serving to connect the Wiener and Feynman integrals. \textit{J. Math. and Phys.} \textbf{39} (1960/61), 126--140.

\bibitem{Dui}
Duistermaat, J. J. Oscillatory integrals, Lagrange immersions and unfolding of singularities. \textit{Comm. Pure Appl. Math.} \textbf{27} (1974), 207--281. 

\bibitem{ELT}
Elworthy, David; Truman, Aubrey. Feynman maps, Cameron-Martin formulae and anharmonic oscillators. \textit{Ann. Inst. H. Poincaré Phys. Théor.} \textbf{41} (1984), no. 2, 115--142.

 \bibitem{evans} Evans, Lawrence C. \textit{Partial Differential Equations.} Second edition. Graduate Studies in Mathematics, 19. American Mathematical Society, Providence, RI, 2010. 

 \bibitem{fei_11} Feichtinger, Hans G.; Huang, Chunyan; Wang, Baoxiang. Trace operators for modulation, $\alpha$-modulation and Besov spaces. \textit{Appl. Comput. Harmon. Anal.} \textbf{30} (2011), no. 1, 110--127. 

 \bibitem{FNT} Feichtinger, Hans G.; Nicola, Fabio; Trapasso, S. Ivan. On exceptional times for pointwise convergence of integral kernels in Feynman-Trotter path integrals. In: \textit{Anomalies in partial differential equations}, 293--311, Springer INdAM Ser., 43, Springer, Cham, 2021. 

\bibitem{fei_mod83} Feichtinger, Hans G. Modulation spaces on locally compact abelian groups. In: \textit{Wavelets and Their Applications}, ed. by S. Thangavelu, M. Krishna, R. Radha (Allied Publishers, New
Dehli, 2003), pp. 99--140. Reprint of 1983 technical report, University of Vienna.

\bibitem{feyn1} Feynman, R. P. Space-time approach to non-relativistic quantum mechanics. \textit{Rev. Modern Physics} \textbf{20} (1948), 367--387. 

\bibitem{feyn2} Feynman, R. P. Space-time approach to quantum electrodynamics. \textit{Phys. Rev.} (2) \textbf{76} (1949), 769--789.

\bibitem{fuji_book} Fujiwara, Daisuke. \textit{Rigorous Time Slicing Approach to Feynman Path Integrals.} Mathematical Physics Studies. Springer, Tokyo, 2017. 

 \bibitem{gro_book} Gr\"ochenig, Karlheinz. \textit{Foundations of Time-Frequency Analysis.} Birkh\"auser Boston, Inc., Boston, MA, 2001. 

 \bibitem{gro_sjo} Gr\"ochenig, Karlheinz. Time-frequency analysis of Sj\"ostrand's class. \textit{Rev. Mat. Iberoam.} \textbf{22} (2006), no. 2, 703--724. 
 
 \bibitem{Hoc}
Hochberg, Kenneth J. A signed measure on path space related to Wiener measure. \textit{Ann. Probab.}\textbf{ 6} (1978), no. 3, 433--458. 

\bibitem{Hor71}
H\"ormander, Lars. Fourier integral operators. I. \textit{Acta Math.} \textbf{127} (1971), no. 1-2, 79--183. 

\bibitem{HorBook}
H\"ormander, Lars. \textit{The Analysis of Linear Partial Differential Operators. I. Distribution theory and Fourier analysis}. Reprint of the second (1990) edition. Classics in Mathematics. Springer-Verlag, Berlin, 2003.


 \bibitem{mazz_book} Mazzucchi, Sonia. \textit{Mathematical Feynman Path Integrals and their Applications.} Second edition. World Scientific Publishing Co. Pte. Ltd., Hackensack, NJ, 2022. 

\bibitem{NT_book} Nicola, Fabio; Trapasso, S. Ivan. \textit{Wave Packet Analysis of Feynman Path Integrals.} Lecture Notes in Mathematics, 2305. Springer, 2022.

\bibitem{NT_cmp} Nicola, Fabio; Trapasso, S. Ivan. On the pointwise convergence of the integral kernels in the Feynman-Trotter formula. \textit{Comm. Math. Phys.} \textbf{376} (2020), no. 3, 2277--2299. 

\bibitem{NT_jmp} Nicola, Fabio; Trapasso, S. Ivan. Approximation of Feynman path integrals with non-smooth potentials. \textit{J. Math. Phys.} \textbf{60} (2019), no. 10, 102103, 13 pp. 

\bibitem{reich-s} Reich, Maximilian; Sickel, Winfried. Multiplication and composition in weighted modulation spaces. In: \textit{Mathematical analysis, probability and applications—plenary lectures}, 103--149, Springer, 2016. 

\bibitem{sjo} Sj\"ostrand, Johannes. An algebra of pseudodifferential operators. \textit{Math. Res. Lett.} \textbf{1} (1994), no. 2, 185--192. 

\bibitem{stein} Stein, Elias M. \textit{Harmonic Analysis: Real-Variable Methods, Orthogonality, and Oscillatory Integrals.} With the assistance of Timothy S. Murphy. Princeton Mathematical Series, 43. Monographs in Harmonic Analysis, III. Princeton University Press, Princeton, NJ, 1993. 

\bibitem{voigt} Voigt, J\"urgen. Factorization in some Fr\'echet algebras of differentiable functions. \textit{Studia Math.} \textbf{77} (1984), no. 4, 333--348.

\bibitem{Tho}
Thomas, Erik. Projective limits of complex measures and martingale convergence. \textit{Probab. Theory Related Fields} \textbf{119} (2001), no. 4, 579--588. 

\bibitem{T_imrn} Trapasso, S.\ Ivan. On the convergence of a novel time-slicing approximation scheme for Feynman path integrals. \textit{Int. Math. Res. Not. IMRN} \textbf{2023}, no. 14, 11930--11961. 

\bibitem{Yam}
Yamasaki, Yasuo. \textit{Measures on Infinite-dimensional Spaces.} Series in Pure Mathematics, 5. World Scientific Publishing Co., Singapore, 1985. 


    \end{thebibliography}
\end{document}